\documentclass[12pt,reqno]{amsart}

\usepackage{graphicx} 
\usepackage{amsmath, amsthm, amssymb,indentfirst}

\usepackage{amsmath, amsthm, amsopn, amssymb, enumerate}

\usepackage{amsmath , amssymb, latexsym }
\usepackage{graphics}
\usepackage{graphicx}
\usepackage{verbatim} 
\usepackage{color}
\usepackage[mathscr]{eucal}
\usepackage{pgf,tikz}
\usepackage{mathrsfs}
\usetikzlibrary{arrows}
\usepackage[font=small,labelfont=bf]{caption}

\theoremstyle{plain}
\newtheorem{theorem}{Theorem}[section]

\newtheorem*{corollary*}{Corollary}
\newtheorem{prop}[theorem]{Proposition}
\newtheorem{lemma}[theorem]{Lemma}
\newtheorem*{proposition*}{Proposition}
\newtheorem*{theorem*}{Theorem}
\newtheorem*{lemma*}{Lemma}

\newtheorem*{claim*}{Claim}

\theoremstyle{definition}

\newtheorem*{definition*}{Definition}

\theoremstyle{remark}

\newtheorem*{obs*}{Observation}

\newtheorem{remark}{Remark}

\newcommand{\Z}[1]{\mathbb{Z}_{#1}}
\newcommand{\pr}[1]{\mathbb{P}\left(#1\right)}

\def\ZZ{\mathbb{Z}}

\newcommand{\Bin}{\ensuremath{\mathrm{Bin}}}

\newcommand{\eps}{\varepsilon}
\newcommand{\Ex}[1]{\mathbb{E}\left[#1\right]}

\newcommand{\eq}[1]{\begin{equation}\label{eq:#1}}
\newcommand{\eqe}{\end{equation}}
\newcommand{\eqr}[1]{\eqref{eq:#1}}

\def\HH{\mathcal{H}}
\def\N{\mathbb{N}}

\def\Z{\mathbb{Z}}

\def\ZZ{\mathbb{Z}}
\def\FF{\mathbb{F}}

\def\Pr{\mathbb{P}}

\def\eps{\varepsilon}
\def\le{\leqslant}
\def\ge{\geqslant}
\def\leq{\leqslant}
\def\geq{\geqslant}
\def\Bin{\textup{Bin}}

\def\<{\langle}
\def\>{\rangle}
\def\0{\textbf{0}}
\def\x{\mathbf{x}}
\def\y{\mathbf{y}}

\def\PP{\mathcal{P}}
\def\valpha{\boldsymbol{\alpha}}

\title[Deviation probabilities for arithmetic progressions]{Deviation probabilities for arithmetic progressions and other regular discrete structures}

\author{Gonzalo Fiz Pontiveros}
\address{Gonzalo Fiz Pontiveros, Department of Mathematics, Universitat Polit\`ecnica de Catalunya, Barcelona, Spain.}
\email{gonzalo.fiz@upc.edu}

\author{Simon Griffiths}
\address{Simon Griffiths, Departamento de Matem\'atica, PUC-Rio, Rua Marqu\^es de S\~ao Vicente 225, G\'avea, Rio de Janeiro 22451-900, Brazil}		
\email{simon@mat.puc-rio.br}

\author{Matheus Secco}
\address{Matheus Secco, Departamento de Matem\'atica, PUC-Rio, Rua Marqu\^es de S\~ao Vicente 225, G\'avea, Rio de Janeiro 22451-900, Brazil}
\email{matheussecco@mat.puc-rio.br}

\author{Oriol Serra}
\address{Oriol Serra, Department of Mathematics, Universitat Polit\`ecnica de Catalunya, and Barcelona Graduate School of Mathematics, Barcelona, Spain.} 
\email{oriol.serra@upc.edu}

\thanks{Acknowledgements: G.F.P. was supported by BGSMath Postdoctoral Grant and the Spanish Research Agency under projects MDM-2014-0445 and MTM2017-82166,  S.G. was supported by CNPq bolsa de produtividade em pesquisa (Proc. 310656/2016-8) and FAPERJ Jovem cientista do nosso estado (Proc. 202.713/2018), M.S. was supported as a PhD student by CAPES and O.S. was supported by grants MDM-2014-0445 and MTM2017-82166-P of the Spanish Ministry of Science.}

\begin{document}
\begin{abstract} 
Let the random variable $X\, :=\, e(\HH[B])$ count the number of edges of a hypergraph $\HH$ induced by a random $m$ element subset $B$ of its vertex set.  Focussing on the case that $\HH$ satisfies some regularity condition we prove bounds on the probability that $X$ is far from its mean.  It is possible to apply these results to discrete structures such as the set of $k$-term arithmetic progressions in the cyclic group $\mathbb{Z}_N$.  Furthermore, we show that our main theorem is essentially best possible and we deduce results for the case $B\sim B_p$ is generated by including each vertex independently with probability $p$.
\end{abstract}
\maketitle
\section{Introduction}

The problem of determining how well a random variable $X$ is concentrated around its mean $\Ex{X}$ has a long history and is of great interest in many areas of mathematics.  In the area of Combinatorics this question most frequently arises when $X$ counts the number of occurrences of some substructure.  The cases of subgraphs in a random graph and arithmetic progressions in a random set have been studied extensively in recent years, as we discuss below.  There is a common framework in which these problems may be viewed: Given a hypergraph $\HH$ on $[N]:=\{1,\dots ,N\}$ and a distribution on subsets $B\subseteq [N]$ we may ask how well concentrated is 
\[
X\, :=\, e(\HH[B])\, ,
\]
the number of edges of $\HH$ contained in $B$. 

Work on this problem has focussed on the case that $B\sim B_p$ is a $p$--random subset of $[N]$ (each element included independently with probability $p$).  In this case general bounds on deviation probabilities of the random variable $X:=e(\HH[B_p])$ follow from the famous inequality of Kim and Vu~\cite{KV2000}.  In the case of large deviations (of the order of the mean) further progress has been made by Janson and Ruci\'nski~\cite{JR2011}, who determined (under certain conditions) the log probability $\log(\pr{X>(1+\delta)\Ex{X}})$ up to a factor of order $\log{1/p}$, and by Warnke~\cite{W2017, W2020} who determined the log probability up to a constant factor (in which the constant may depend on $\delta$).  Recently Bhattacharya and Mukherjee~\cite{BM2019} have shown how the large deviations framework introduced by Chatterjee and Varadhan~\cite{CV2011} (in the context of subgraph counts) may be used to understand when replica symmetry is broken.


We focus on the case $B\sim B_m$ is a uniformly random $m$ element subset of $[N]=\{1,\dots ,N\}$.  We prove upper bounds on the probability of deviations of the random variable $e(\HH[B_m])$, which are particularly relevant in the case that the hypergraph $\HH$ is highly regular.  Furthermore, since $B_p$, the $p$-random subset, may be obtained by averaging over the models $B_m$ where $m$ is selected according to $\Bin(N,p)$, we may deduce results in the context of $p$--random sets.

Given a $k$-uniform hypergraph $\HH$ on vertex set $[N]$ and a subset $B\subseteq [N]$ we set
\[
N^{\HH}(B) \, :=\, e(\HH[B])\, ,
\]
the number of edges of $\HH$ contained in the set $B$.  We also define
\[
L^{\HH}(m)\, :=\, \Ex{N^{\HH}(B_m)}\, ,
\]
 the expected value of $N^{\HH}(B_m)$ where $B_m$ is a uniformly selected $m$ element subset of $[N]$.  Our focus will be on studying
 \[
 D^{\HH}(B_m)\, :=\, N^{\HH}(B_m)\, -\, L^{\HH}(m)\, ,
 \]
 the deviation of $N^{\HH}(B_m)$ from its mean.
 
We may now state our bounds on the probability of certain deviations $D^{\HH}(B_m)$.  We say that a hypergraph $\HH$ is $r$--tuple--regular\footnote{Equivalently, in the language of combinatorial designs, $\HH$ is an  $r-(N,k,\lambda)$ design, for some $\lambda$ .} if all $r$--tuples of vertices are included in the same number of edges.  Note that all hypergraphs are $0$--tuple-regular, and a hypergraph is $1$--tuple-regular if it is regular in the usual sense that each vertex is in the same number of edges.  We also note that, by a simple double counting argument, $r$--tuple regular implies $r'$--tuple regular for all $0\le r'\le r$.

\begin{theorem}\label{thm:reg}.  Let $1\le r\le k$.
Let $\HH$ be a $k$--uniform hypergraph on $[N]$.  Suppose that $\HH$ is $(r-1)$--tuple--regular with maximum $r$-degree $\Delta_r$.  Then
\[
\pr{|D^{\HH}(B_m)|>a}\, \le\, N^{O_k(1)}\, \exp\left(\frac{-\Omega_k(1) a^{2/r}}{m\Delta_{r}^{2/r}}\right)
\]
for all $a>0$.
\end{theorem}

\begin{remark} We focus throughout the article on hypergraphs rather than even more general structures such as weighted hypergraphs or polynomials, that were considered by Kim and Vu~\cite{KV2000}.  In fact, the proof goes through with no major changes for the case of hypergraphs with positively weighted edges.  Since a general weighted hypergraph may be written as a difference of two with positive weights, one may deduce results in this setting.  Similarly, a bounded degree polynomial may be broken into a finite number of homogeneous polynomials which correspond to weighted hypergraphs.
\end{remark}

Unfortunately some natural hypergraphs are not precisely $r$-tuple-regular, but just very close to being so.  Let $\bar{d}_r=\bar{d}_r(\HH)$ denote the average degree of $r$-sets in a hypergraph $\HH$.  We say that $\HH$ is $(r,\eta)$--near--regular if every $r$--tuple of vertices is contained in $(1\pm \eta)\bar{d}_r$ edges.  Obviously, a hypergraph which is $r$-tuple-regular is $(r,\eta)$-near-regular for all $\eta\ge 0$.  In particular, Theorem~\ref{thm:reg} is a special case ($\eta=0$) of the following theorem.

\begin{samepage}
\begin{theorem}\label{thm:nearreg}  Let $1\le r\le k$ and let $\eta\in [0,3^{-r+1}]$.
Let $\HH$ be a $k$--uniform hypergraph on $[N]$.  Suppose that $\HH$ is $(r-1,\eta)$--near--regular with maximum $r$--degree $\Delta_r$.  Then
\[
\pr{|D^{\HH}(B_m)|>a}\, \le\, N^{O_k(1)}\, \exp\left(\frac{-\Omega_k(1) a^{2/r}}{m\Delta_{r}^{2/r}}\right)
\]
for all
\[
a\, \ge\, \left( 10k! \right)^{10^r}e(\HH)\left(\frac{\eta m^{k-1} }{N^{k-1}}\right)^{r/(r-1)}.
\]
\end{theorem}
\end{samepage}

\begin{remark} In the dense case ($m=\Theta(N)$), the bound given by Theorem~\ref{thm:nearreg} is best possible (up to the constant implicit in $\Omega_k(1)$) for all $r\ge 0$ and across the whole range $\Theta(1)\le \Delta_r\le \Theta(N)$.  Examples are given in Section~\ref{sec:LB}.
\end{remark}

\begin{remark}\label{rem:minr} One may easily observe by a double counting argument that an $(r-1,\eta)$-near-regular hypergraph $\HH$ is $(r'-1,\eta)$-near-regular for all $1\le r'\le r$.  So one may choose which of the above inequalities to apply.  Therefore
\[
\pr{D^{\HH}(B_m)>a}\, \le\, N^{O_k(1)}\, \min_{1\le r'\le r} \left\{\exp\left(\frac{-\Omega_k(1) a^{2/r'}}{m\Delta_{r'}^{2/r'}}\right)\right\}\, .
\]
It is worth remarking that the minimum is not always obtained at the extremes $r'\in \{1,r\}$, see the application to arithmetic progressions in Section \ref{sec:applications} for example.
\end{remark}

\begin{remark}\label{rem:regand}  The lower bound condition on $a$ given in Theorem~\ref{thm:nearreg} is best possible in the weak sense that there exist hypergraphs for which it cannot be significantly improved.  However, with extra conditions one may weaken the condition on $a$.  In particular, the condition in Theorem~\ref{thm:nearreg} may be weakened to
\[
a\, \ge\, \left( 10k! \right)^{10^r}e(\HH)\left(\frac{\eta m^{k-2} }{N^{k-2}}\right)^{r/(r-2)}
\]
if $\HH$ is regular and $(r-1,\eta)$-near-regular, for $r\ge 3$.  This result is stated as Proposition~\ref{prop:regand} in Section~\ref{sec:regand}. \end{remark}


The existing literature has focussed on the probability of deviations of $e(\HH(B_p))$ in the model of a $p$-random subset $B_p\subseteq [N]$, in which each element is included in $B_p$ independently with probability $p$.  As usual in probabilistic combinatorics we allow for the possibility that $p=p(n)$ is a function of $n$.  Our results apply for a range of moderate deviations, which are smaller than the order of magnitude of the mean.

For the range of deviations covered by the theorem below  it precisely provides the asymptotic log probability of the corresponding deviation event.   For comparison, the bounds given by the  Kim-Vu inequality~\cite{KV2000} are not so strong, however they often apply for a much larger range of deviations.

\begin{theorem}\label{thm:pworld}
Let $k\ge r\ge 2$.  Let $\HH_N$ be a sequence of $k$-uniform hypergraphs which are $(r-1,\eta)$-near-regular with maximum $r$-degree $\Delta_r$ and $V(\HH_N)=[N]$.  Let $\delta_N$ be a sequence satisfying
$$
\max \left \{\frac{\Delta_r (N\log N)^{r/2}}{p^{k-r/2} e(\HH)}\, ,\, \big(\eta^{r} p^{k-r}\big)^{1/(r-1)}\, ,\, \frac{1}{\sqrt{pN}} \right \} \ll \delta_N \ll \left(\frac{p^{k-r}e(\HH)}{N^r\Delta_r}\right)^{1/(r-1)}\, ,
$$
where $0 \leq p \leq 1$ is bounded away from $1$.
Then
$$
\pr{|D^{\HH_{N}} (B_p)| > \delta_N L^{\HH_{N}}(B_p)} \, =\,  \exp \left( -(1+o(1))\frac{\delta_N^2 pN}{2k^2(1-p)}\right)\, .
$$
\end{theorem}

\begin{remark}\label{rem:regp} As in Remark~\ref{rem:regand}, we may be able extend the range of $\delta_N$  if we also know that $\HH$ is regular as well as $(r-1,\eta)$-near-regular.  Using Proposition~\ref{prop:regand} in place of Theorem~\ref{thm:nearreg} one may replace the term $\big(\eta^{r} p^{k-r}\big)^{1/(r-1)}$ by $\big(\eta^{r}p^{2k-2r}\big)^{1/(r-2)}$ in the lower bound of $\delta_N$ in Theorem \ref{thm:pworld}.
\end{remark}

When $\Delta_r$ is much larger than $e(\HH)/N^r$, the order of magnitude of the average $r$-degree, the range of deviations covered by the above theorem may be disappointing.  In particular this theorem does not appear to give new results when applied in the context of subgraph counts in the Erd\H os-R\'enyi random graph $G(n,p)$.  For results concerning moderate deviations for subgraph counts the reader may consult~\cite{DE2009,FMN2016,GGS2019,JW2015}.  There is a large literature dedicated to problems of large deviations of subgraph counts, we encourage the interested reader to consult the survey of Chatterjee~\cite{C2016}, the recent article of Harel, Mousset and Samotij~\cite{HMS2019}, and the references therein.

On the other hand the condition on $\delta_N$ in Theorem~\ref{thm:pworld} simplifies significantly if $\Delta_r$ is of the same order as the average $r$-degree.  For example, when $r=2$ and the hypergraph is regular the condition simplifies to
\[
\max \left \{\frac{\log{N}}{p^{k-1}N}\, ,\, \frac{1}{\sqrt{pN}} \right \} \ll \delta_N \ll p^{k-2}\, .
\]

\newpage
As an immediate application, consider $3$-term arithmetic progressions in the cyclic group $\Z/\N\Z$, for $N$ prime.  Let us write $D^{3}(B_p)$ for the deviation of the $3$--progressions count in a $p$-random subset $B_p$ of $\Z/\N\Z$.  Note that the expected number of such arithmetic progressions is $L^{3}(B_p)=p^3\binom{N}{2}$.

\begin{samepage}

\begin{theorem}\label{thm:APpworld}
Let $\delta_N$ be a sequence satisfying
\[
\max \left \{\frac{\log{N}}{p^{2}N}\, ,\, \frac{1}{\sqrt{pN}} \right \}\, \ll \,\delta_N\, \ll\, p.
\]
Then,
\[
\pr{D^{3} (B_p) > \delta_N L^3(B_p)} \, =\,  \exp \left( -(1+o(1))\frac{\delta_N^2 pN}{18(1-p)}\right)\, .
\]
Furthermore, the same bounds apply to the corresponding negative deviations.
\end{theorem}

\end{samepage}

See Section~\ref{sec:applications} for a discussion of Theorem \ref{thm:APpworld} and analogous results.  

We remark that the results above would not hold for hypergraphs without some kind of regularity property.  Related results for the non-regular setting, which includes arithmetic progressions in $[N]=\{1,\dots ,N\}$ for example, have recently been obtained by Christoph Koch together with the second and third authors~\cite{GKS2020}.

\subsection{Layout of the article}

In Section~\ref{sec:overview} we prepare for the proof of Theorem~\ref{thm:nearreg} by stating a number of auxiliary results which we require.  In particular we state a martingale representation for the deviation $D^{\HH}(B_m)$, which may be of independent interest.  This section also introduces notation and required inequalities from probability theory.

In Section~\ref{sec:mproof} we complete the proof of Theorem~\ref{thm:nearreg}.   This is achieved by a double induction argument which links deviation probabilities to the increments of the martingale and vice-versa in successive steps. In Section~\ref{sec:pworld} we show how Theorem~\ref{thm:pworld} may be deduced from Theorem~\ref{thm:nearreg}. 

Section~\ref{sec:applications} illustrates possible applications of Theorem~\ref{thm:nearreg}  to obtain deviation probabilities for a variety of arithmetic structures. These include $k$--arithmetic progressions, Schur triples, additive energy and, more generally, solutions of linear systems in random sets.  We also include a direct proof of the deviation result for $3$-term arithmetic progressions which is simpler than the general case (Theorem~\ref{thm:nearreg}).  Since many of the same ideas arise, the reader may wish to read this proof before that of Theorem~\ref{thm:nearreg}.  See Theorem~\ref{thm:3apm}, and its proof.

In Section~\ref{sec:LB} we prove that the bound given in Theorem~\ref{thm:nearreg} is best possible up to the implicit constant $\Omega_k(1)$.  The proof is based on a construction of hypergraphs with certain regularity properties.  The details of the construction are given in the appendix.

Finally, in Section~\ref{sec:final} we give some concluding remarks and discuss some related open questions.






\section{Overview and Auxiliary results}\label{sec:overview}

In this section we introduce notation and auxiliary results which we require for the proof of Theorem~\ref{thm:nearreg}.  In particular we introduce a martingale representation for $D^{\HH}(B_m)$ (see Proposition~\ref{prop:Mart}) which is the basis of our proof of Theorem~\ref{thm:nearreg}.  This representation is similar in spirit to that introduced in~\cite{GGS2019} in the context of subgraph counts.  The section is divided into Section~\ref{sec:notation} in which we introduce notation, Section~\ref{sec:Aux} in which we state auxiliary results and Section~\ref{sec:prob} in which we state some inequalities from probability theory.  The proofs of the auxiliary results of Section~\ref{sec:Aux} are presented in Section ~\ref{sec:Auxproofs}.

\subsection{Notation}\label{sec:notation}

We will continue to use $N$ for the number of vertices of the hypergraph $\HH$.  We set $h:=e(\HH)$.

We will continue to use $m$ for the number of elements in the random subset $B_m\subseteq [N]$.  We denote its density by $t:=m/N$.  We sometimes consider the set $B_i$ at an earlier step $i$ of the process.  We denote its density $s:=i/N$.

Given a vertex $x$ of a hypergraph $\HH$ we define
\[
\HH(x)\, :=\, \{f\setminus \{x\}\, :\, f\in E(\HH)\, , \, x\in f\}\, .
\]
In the case that $\HH$ is a $k$-uniform hypergraph on $[N]$, then $\HH(x)$ is a $(k-1)$-uniform hypergraph on $[N]\setminus \{x\}$.

In addition to $N^{\HH}(B_m)$, the number of edges of $\HH$ in $B_m$, we shall also consider partially filled edges (with multiplicity).  Let $N_j^{\HH}(B_m)$ be the number of $j$--subsets of edges of $\HH$ which are contained in $B_m$. Equivalently,
\[
N_j^{\HH}(B_m)\, :=\, \sum_{f\in E(\HH)} \binom{|f\cap B_m|}{j}\, .
\]
For a $k$-uniform hypergraph $\HH$ we have $N^{\HH}(B_m)=N^{\HH}_{k}(B_m)$.


Throughout the article we denote by
\begin{align}
&L_{j}^{\HH}(m)\, :=\, \Ex{N_{j}^{H}(B_m)}\, =\, h\binom{k}{j} \frac{(m)_j}{(N)_j}\, , \quad \text{and}\label{eq:Lj}\\
&D_j^{\HH}(B_m)\, := \, N_j^{\HH}(B_m)\, -\, L_j^{\HH}(m)\, ,\label{eq:Dj}
\end{align}
the mean and deviation of $N_j^{\HH}(B_m)$ respectively.

In order to define the increments of the key martingale representation, we introduce\footnote{Implicit in this notation is the fact that we regard $B_i$ as a set together with the information of the order in which the points were added.}
\[
X^{\HH}_{\ell}(B_i)\, :=\, N^{\HH}_{\ell}(B_i)\, -\, \Ex{N^{\HH}_{\ell}(B_i)\, |\,B_{i-1}}.
\]
Since $N^{\HH}_{\ell}(B_{i-1})$ is determined by $B_{i-1}$ we observe that
\[
X^{\HH}_{\ell}(B_i)\, =\, A^{\HH}_{\ell}(B_i)\, -\, \Ex{A^{\HH}_{\ell}(B_i)\, |\,B_{i-1}},
\]
where 
\[
 A^{\HH}_{\ell}(B_i)\, :=\, N^{\HH}_{\ell}(B_i)\, -\, N^{\HH}_{\ell}(B_{i-1})
\]
denotes the increase in $N^{\HH}_{\ell}(B_i)$ with the addition of the $i$--th element.  Both the above  expressions for $X^{\HH}_{\ell}(B_i)$ will be used during the proof. We observe that, by its definition, the sequence $X^{\HH}_{\ell}(B_i)$ is a difference martingale for each $\ell$. 

When the hypergraph $\HH$ is clear from context we drop it from the notation.  That is, we write $N_{j}(B_m), D_j(B_m), A_{\ell}(B_i), X_{\ell}(B_i)$, etc.

It will sometimes be convenient to write $x\in a\pm b$ to express that $x$ belongs to the interval $[a-b,a+b]$.

\subsection{Auxiliary results}\label{sec:Aux}

The most significant auxiliary result we state here is the martingale representation for $D^{\HH}_j(B_m)$.

\begin{prop}\label{prop:Mart}
Let $\HH$ be a $k$-uniform hypergraph and let $1\le j\le k$. Then
\begin{equation}\label{eq:Mart}
D^{\HH}_{j}(B_m)\, =\, \sum_{i=1}^{m} \sum_{\ell=1}^{j} \frac{(N-m)_{\ell}(m-i)_{j-\ell}}{(N-i)_j}\, \binom{k-\ell}{k-j}\, \, X^{\HH}_{\ell}(B_i)\, .
\end{equation}
\end{prop} 

The equation \eqref{eq:Mart} expresses $D^{\HH}_{j}(B_m)$ as a linear combination of difference martingales and it is therefore a martingale. 

It will also be useful to note the ways in which $\eta$-near-regularity is inherited.

\begin{lemma}\label{lem:nearreg}
Let $\HH$ be a $k$-uniform $(r,\eta)$-near-regular hypergraph with maximum $(r+1)$-degree $\Delta_{r+1}$, for some $\eta \in [0, 1/3]$.  Let $x\in V(\HH)$.  Then
\begin{enumerate}
\item[(i)] $\HH$ is $(r-1,\eta)$-near-regular
\item[(ii)] $\HH(x)$ is $(r-1, 3\eta)$-near-regular with maximum $r$-degree at most $\Delta_{r+1}$.
\end{enumerate}
\end{lemma}

A major part of the proof of Theorem~\ref{thm:nearreg} involves controlling the size of the increments in the martingale representation.  In some cases we can control $|X^{\HH}_{\ell}(B_i)|$ directly and deterministically.

\begin{lemma}\label{lem:deterministic}
Let $\HH$ be a $k$-uniform hypergraph on $[N]$ which is $(r,\eta)$-near-regular. If $1 \leq \ell \leq r$, then
\[
|X_{\ell}^{\HH}(B_i)| \, \leq \, \frac{2 \ell \binom{k}{\ell} \eta s^{\ell-1} h}{N}. 
\]
\end{lemma}

Finally, given a $B_i$--measurable event $E$ let $E^{-}$ be the $B_{i-1}$--measurable event that $E$ occurs for some extension $B_{i-1}\cup \{x\}$ of $B_{i-1}$.  Since the conditional probability $\pr{E|E^{-}}$ is between $1/N$ and $1$ we immediately obtain the following.

\begin{lemma}\label{lem:pN} Let $E$ be a $B_i$ measurable then
\[
\pr{E}\, \le\, \pr{E^{-}}\, \le\, N\pr{E}\, .
\]
\end{lemma}

\subsection{Martingale Inequalities}\label{sec:prob} 

Our proofs will rely on the Azuma--Hoeffding~\cite{A1967,H1963} martingale inequality.
 
\begin{lemma}[Azuma--Hoeffding inequality]\label{lem:HA}
Let $(S_i)_{i=0}^{m}$ be a martingale with increments $(X_i)_{i=1}^{m}$, and let $c_i=\|X_i\|_{\infty}$ for each $1\le i\le m$.
Then, for each $a>0$, 
\[
\pr{S_m-S_0\, >\, a}\, \le \, \exp\left(\frac{-a^2}{2\sum_{i=1}^{m}c_i^2}\right)\, .
\]
Furthermore, the same bound holds for $\pr{S_m-S_0\, <\, -a}$.
\end{lemma}
 
By considering a ``truncation'' of the increments in which $X_i$ is set to $0$ if  it could be larger than $c_i$ with positive probability then one immediately obtains the following straightforward variant.  We state the lemma in our context of a sequence of random sets $(B_i)_{i=0}^{N}$.  This process may be defined by taking $b_1,\dots,b_N$ to be a uniformly random permutation of $\{1,\dots ,N\}$ and setting $B_i=\{b_1,\dots ,b_i\}$ for each $i=0,\dots N$.

\begin{lemma}[Azuma--Hoeffding inequality (a variant)]\label{lem:HAv}
Let $(S_i)_{i=0}^{m}$ be a martingale with respect to the natural filtration of the process $(B_i)_{i=0}^{N}$, let $(X_i)_{i=1}^{m}$ be the increments of the process and let $(c_i)_{i=1}^{N}$ be a sequence of real numbers.
Then, for each $a>0$, 
\[
\pr{S_m-S_0\, >\, a}\, \le \, \exp\left(\frac{-a^2}{2\sum_{i=1}^{m}c_i^2}\right)\, +\, N\sum_{i=1}^{m}\pr{|X_i|>c_i}\, .
\]
Furthermore, the same bound holds for $\pr{S_m-S_0\, <\, -a}$.
\end{lemma} 
 
\begin{proof}
We first define for $1 \leq i \leq m$ the ``truncation'' $X^*_i$ of the increment $X_i$ as
\[
X^*_i\, :=\, X_i\, 1_{\|X_i|B_{i-1}\|_{\infty}\le c_i}\, .
\]
Let us define a new process $(S^*_j)_{j=0}^{m}$ by $S^*_0 := S_0$ and for $1 \leq j \leq m$: 
\[
S^*_j := S^*_0 + \sum_{i=1}^j X^*_i.
\]
Since the event considered by the indicator function is $B_{i-1}$--measurable we have
$\Ex{X^*_i\, |\, B_{i-1}}\, =\, 0$
and so $(S^*_j)_{j=0}^{m}$ is a martingale with respect to the natural filtration of the process $(B_i)_{i=0}^{N}$. Note also that the increments of this process satisfy $|X^*_i| \le c_i$ almost surely. Therefore the Azuma--Hoeffding inequality gives us
\begin{equation}\label{eq:ha_truncation}
\pr{S^*_m - S^*_0 \, >\, a} \, \leq \, \exp\left(\frac{-a^2}{2\sum_{i=1}^{m}c_i^2}\right).
\end{equation}
We also observe that by union bound and Lemma ~\ref{lem:pN}

\begin{equation}\label{eq:dif_truncation}
\begin{split}
\pr{S^*_m \neq S_m}\,  &\leq\, \sum_{i=1}^{m} \pr{X^*_i \neq X_i} \\
                    &=\, \sum_{i=1}^{m} \pr{\|X_i|B_{i-1}\|_{\infty} > c_i} \\
                    &\leq \, N \sum_{i=1}^{m} \pr{|X_i| > c_i}.
\end{split}
\end{equation}
Finally we have 
\[
\pr{S_m-S_0\, >\, a} \, \le \, \pr{S^*_m - S^*_0 \, >\, a} + \pr{S^*_m \neq S_m}
\]
and so by \eqref{eq:ha_truncation} and \eqref{eq:dif_truncation} we get the desired result.
\end{proof}

\subsection{Proofs of auxiliary results}\label{sec:Auxproofs}
 
We now prove Proposition~\ref{prop:Mart} and Lemmas~\ref{lem:nearreg} and~\ref{lem:deterministic}. 

\begin{proof}[Proof of Proposition~\ref{prop:Mart}]   Fix the hypergraph $\HH$.  We prove the required expression~\eqref{eq:Mart} by a double induction over $m$ and $j$.  The base cases $j=0$ and $m=0$ are trivial.

For the induction step we may assume that~\eqref{eq:Mart} holds if $j'<j$ or if $j'=j$ and $m'<m$.
The argument proceeds by focussing on the latest point added.  We recall that $X_j(B_m)=A_j(B_m)-\Ex{A_j(B_m)|B_{m-1}}$ and that $A_j(B_m)=N_{j}(B_m)-N_{j-1}(B_{m-1})$ counts the increase in $N_j(B_m)$ with the addition of the $m$--th element of $B_m$.  Considering that any such increase must consist of a $(j-1)$--subset together with an extra element of the same edge (which is not already present) and each vertex has probability $1/(N-m+1)$ to be selected next, we have that
\[
\Ex{A_j(B_{m})|B_{m-1}}\, =\, \frac{(k-j+1)N_{j-1}(B_{m-1})\, -\, jN_{j}(B_{m-1})}{N-m+1}\, .
\]
We will use this expression to find a suitable expression for $D_{j}(B_m)$ in terms of the deviations $D_{j-1}(B_{m-1})$, $D_{j}(B_{m-1})$ and $X_j(B_m)$.  The first step will be to express $D_j$ as $N_j-L_j$.  To reach the later steps we expand $N_j(B_m)$ as $A_j(B_m)+N_j(B_{m-1})$ and when possible express $N_j$ as $L_j+D_j$ and use the identity
\[
\frac{(k-j+1)L_{j-1}(B_{m-1})\, -\, jL_{j}(B_{m-1})}{N-m+1}\, =\, L_{j}(m)\, -\, L_j(m-1)\, .
\]
We obtain the following expression for $D_{j}(B_m)$:
\begin{align*}
D_j(B_m)\, & =\, N_j(B_m)\, -\, L_j(m)\\
 & =\, A_j(B_m)\, +\, D_j(B_{m-1})\, -\big(L_j(m)\, -\, L_j(m-1)\big)\\
 & =\, \frac{N-m-j+1}{N-m+1}\, D_j(B_{m-1})\, +\, \frac{(k-j+1)}{N-m+1}D_{j-1}(B_{m-1})\, +\, X_j(B_m)\, .
 \end{align*}
The required expression~\eqref{eq:Mart} now follows immediately from the induction hypothesis by simply checking the coefficient of each $X_{\ell}(B_i)$.  This may be verified easily by checking the cases (i) $i=m$ and $\ell=j$, (ii) $i = m$ and $\ell < j$, (iii) $i < m$ and $\ell = j$, and (iv) $i < m$ and $\ell < j$.

In case (i), the coefficients on each side are $1$ and in case (ii), the coefficients on each side are $0$. In case (iii), the coefficients on both sides are $(N-m)_j/(N-i)_j$. Finally in case (iv), the coefficient on the right hand side is given by 
\begin{align*}
&\frac{N-m-j+1}{N-m+1} \cdot \frac{(N-m+1)_{\ell} (m-1-i)_{j-\ell}}{(N-i)_j} \cdot \binom{k - \ell}{k - j} \, \phantom{\Bigg|} \\
&+ \frac{k-j+1}{N-m+1} \cdot \frac{(N-m+1)_{\ell} (m-1-i)_{j-1-\ell}}{(N-i)_{j-1}} \cdot \binom{k - \ell}{k - j +1}  \phantom{\Bigg|} \\
&= \frac{(N-m)_{\ell-1} (m-1-i)_{j-1-\ell}}{(N-i)_j} \binom{k-\ell}{k-j} \left[ (N-m-j+1)(m-i-j+\ell) + (N-i-j+1)(j-\ell)\right]\phantom{\Bigg|}  \\
&= \frac{(N-m)_{\ell-1} (m-1-i)_{j-1-\ell}}{(N-i)_j} \binom{k-\ell}{k-j} \cdot (N-m-\ell+1)(m-i) \phantom{\Bigg|} \\
&= \frac{(N-m)_{\ell}(m-i)_{j-\ell}}{(N-i)_j} \binom{k-\ell}{k-j},\vspace{2mm}\phantom{\Bigg|} 
\end{align*}
which agrees with the coefficient on the left hand side.
\end{proof}

\begin{proof}[Proof of Lemma~\ref{lem:nearreg}]
Part (i) is immediate by a double counting argument.  This argument shows that the average $(r-1)$--degree is $(N-r+1)/r$ times the average $r$--degree, while the maximum $(r-1)$--degree is at most this multiple of the maximum $r$--degree, and similarly for the minimum.

For (ii) we observe that $\HH(x)$ is $(k-1)$-uniform with maximum $r$-degree at most $\Delta_{r+1}$. We will prove now that $\HH(x)$ is $(r-1,3\eta)$-near-regular.   Since $\HH$ is $(r, \eta)$-near-regular, it is also $(1,\eta)$-near-regular, by (i), and therefore $d_{\HH}(x) \, \ge\, (1-\eta) hk/N$.  And so it follows from a simple double counting argument that the average $(r-1)$-degree in $\HH(x)$ satisfies
\[
\bar{d}_{r-1}^{\, \HH(x)}\,  = \, \frac{(k-1)_{r-1}}{(N-1)_{r-1}} \, d_{\HH}(x) \, =\, \bar{d}_r^{\, \HH} \cdot \frac{N}{hk} \, d_{\HH}(x)\,  \ge\, (1-\eta)\bar{d}_r^{\HH}\, .
\]
Now let $A \subseteq V(\HH(x))$ be an $(r-1)$ element set. Since $\HH$ is $(r,\eta)$-near-regular we have
\[
d_{\HH(x)}(A)\, =\, d_{\, \HH} (A \cup \{ x \})\, \le \, (1 + \eta)\bar{d}_r^{\, \HH},
\]
and so, since $\eta\in [0,1/3]$ we have
\[
d_{\HH(x)}(A)\, \leq\, \frac{1 + \eta}{1 -\eta} \,\bar{d}_{r-1}^{\, \HH(x)}\, \le \, (1 + 3\eta)\, \bar{d}_{r-1}^{\, \HH(x)}\, .
\]
A near identical argument gives the lower bound $(1-\eta)/(1+\eta)\ge (1-3\eta)$ times $\bar{d}_{r-1}^{\, \HH(x)}$, and so completes the proof.
\end{proof}
 
 We now prove Lemma~\ref{lem:deterministic}, which bounds the possible value of $|X_{\ell}(B_i)|$ in a $k$-uniform $(r,\eta)$-near-regular hypergraph $\HH$.
 
 \begin{proof}[Proof of Lemma~\ref{lem:deterministic}]
Recall first that $X_{\ell}(B_i):= A_{\ell}(B_i)-\Ex{A_{\ell}(B_i)|B_{i-1}}$ and that $A_{\ell}(B_i) = N_{\ell}(B_i) - N_{\ell}(B_{i-1})$. If $b_i$ is the last element added to $B_i$, i.e, $B_i = B_{i-1} \cup \{b_i\}$, then
\[
A_{\ell}(B_i) = \sum_{\substack{C \subseteq B_{i-1} \\ |C| = \ell -1}} d(C \cup \{b_i\}).
\]
As $1 \leq \ell \leq r$ and $\HH$ is $(r,\eta)$-near-regular, we have that $\HH$ is also $(\ell,\eta)$-near-regular by part (i) of Lemma~\ref{lem:nearreg}.  And so $d(C \cup \{b_i\}) =(1\pm \eta)\bar{d}_{\ell}$ for all $(\ell-1)$ element subsets $C \subseteq B_{i-1}$.  It follows that
\[
A_{\ell}(B_i) \, =\,  (1\pm \eta) \binom{i-1}{\ell-1}\bar{d}_{\ell}\, .
\]
Now, as $X_{\ell}(B_i):= A_{\ell}(B_i)-\Ex{A_{\ell}(B_i)|B_{i-1}}$, it follows that
\[
|X_{\ell}(B_i)|\,  \leq\, 2 \eta \binom{i-1}{\ell-1} \bar{d}_{\ell}\, .
\]

Finally, since $\bar{d}_{\ell} = h \binom{k}{\ell}/\binom{N}{\ell}$, and using the bound $(i-1)_{\ell -1}\le s^{\ell-1}(N-1)_{\ell-1}$, we obtain
\[
|X_{\ell}(B_i)|\, \leq \, \frac{2\ell \binom{k}{\ell} \eta s^{\ell-1} h}{N}\, ,
\]
as required.
\end{proof}


\section{Proof of Theorem~\ref{thm:nearreg}}\label{sec:mproof}

The proof of Theorem~\ref{thm:nearreg} is given by induction on $r$.  In fact we prove two series of statements with a joint induction.  Since $N^{\HH}_{k}(B_m)=N^{\HH}(B_m)$, it is clear that Theorem~\ref{thm:nearreg} is included in the sequence of statements $P_{r}, r\ge 1$.

{\large $\mathbf{P_{r}}$}: \textit{ For all $k\ge j\ge r$ and $\eta\in [0,3^{-r+1}]$, for all $k$-uniform hypergraphs $\HH$ on $[N]$ which are $(r-1,\eta)$-near-regular with maximum $r$-degree $\Delta_r$, we have}
\[
\pr{D^{\HH}_j(B_m)>a}\, \le\, N^{O_k(1)}\, \exp\left(\frac{-\Omega_k(1) a^{2/r}}{m\Delta_{r}^{2/r}}\right) \phantom{\Bigg|}
\]
\textit{for all  $0\le m\le N$ and all $a\ge C_r \eta^{r/(r-1)}ht^{(j-1)r/(r-1)}$, where $C_r = \left( 10k! \right)^{10^r}$.}

The other sequence of statements $Q_r$ will be related to the behaviour of the random variables $X^{\HH}_{\ell}(B_i)$ that occur in the martingale representation of $D_{\HH}(B_m)$.  We define $Q_r$ to be the following statement.

{\large $\mathbf{Q_{r}}$}: \textit{For all $k\ge \ell\ge r+1$ and $\eta\in [0,3^{-r}]$, for all $k$-uniform hypergraphs $\HH$ on $[N]$ which are $(r,\eta)$-near-regular with maximum $(r+1)$-degree $\Delta_{r+1}$, we have}
\[
\pr{|X^{\HH}_{\ell}(B_i)|\, >\, \alpha}\, \le\, N^{O_k(1)}\, \exp\left(\frac{-\Omega_k(1) \alpha^{2/r}}{i\Delta_{r+1}^{2/r}}\right)
\]
\textit{for all $0\le i\le N$ and all $\alpha\ge D_r\eta s^{\ell -1}h/N$, where $D_r = \left(10k!\right)^{10^r+10}$.}

The base case of the induction, $P_1$, is proved in Section~\ref{sec:base} by a straightforward application of the Azuma--Hoeffding inequality.  We complete the proof by showing that $P_r$ implies $Q_r$ and $Q_r$ implies $P_{r+1}$.  These proofs are given in Sections~\ref{sec:PtoQ} and~\ref{sec:QtoP} respectively.

\subsection{The base case -- $P_1$} \label{sec:base}

Let us fix $k\ge j\ge 1$, $\eta\in [0,1]$, and a $k$-uniform hypergraph $\HH$ on $[N]$ with maximum degree $\Delta$.

We use the Azuma--Hoeffding inequality applied to the martingale representation:
\[
D_{j}(B_m)\, =\, \sum_{i=1}^{m} \sum_{\ell=1}^{j} \frac{(N-m)_{\ell}(m-i)_{j-\ell}}{(N-i)_j}\, \cdot \, \binom{k-\ell}{k-j}\, X_{\ell}(B_i)\, .
\]
In order to do so we must bound the magnitude of the increment
\[
Y_i\, :=\,  \sum_{\ell=1}^{j} \frac{(N-m)_{\ell}(m-i)_{j-\ell}}{(N-i)_j}\, \cdot \, \binom{k-\ell}{k-j}\, X_{\ell}(B_i)
\]
of the martingale.  We observe that the first fraction is always at most $1$, and so the coefficient itself is $O_k(1)$.  Recalling that $X_{\ell}(B_i):= A_{\ell}(B_i)-\Ex{A_{\ell}(B_i)|B_{i-1}}$, and that both $A_{\ell}(B_i)$ and $\Ex{A_{\ell}(B_i)|B_{i-1}}$ are non-negative we have
\[
\|X_{\ell}(B_i)\|_{\infty}\, \le\, \|A_{\ell}(B_i)\|_{\infty} \, \le\, O_k(1)\, \cdot \, \Delta\qquad a.s.
\]
where the second inequality follows since $A_{\ell}(B_i)=N_{\ell}(B_i)-N_{\ell}(B_{i-1})$ is certainly at most $\binom{k-1}{\ell-1}=O_k(1)$ times $d_{\HH}(b_i)\le \Delta$ (any ``new'' $\ell$-sets must be in edges containing $b_i$).

Since $Y_i$ consists of a finite number of terms and the coefficients are $O_k(1)$,
\[
|Y_i|\, \le\, O_k(1)\, \cdot\, \Delta \qquad a.s.
\]
By an application of the Azuma--Hoeffding inequality to $D_{j}(B_m)=\sum_{i=1}^{m}Y_i$, we have that
\[
\pr{D_j(B_m)\, >\, a}\, \le\, \exp\left(\frac{-a^2}{2m \, O_k(1)\, \Delta^2}\right)\, =\, \exp\left(\frac{-\Omega_k(1) a^2}{m\Delta^2}\right)\, .
\]
This complete the proof of the base case $P_1$.

\subsection{P implies Q}\label{sec:PtoQ}

In this section, we shall prove that $P_r\, \Rightarrow\, Q_r$. Let us fix $k\ge \ell\ge r+1$, $\eta\in [0,3^{-r}]$ and a $k$-uniform hypergraph $\HH$ on $[N]$ which is $(r,\eta)$-near-regular.  Let $\Delta_{r+1}$ be the maximum $(r+1)$-degree of $\HH$.  Let us also fix $0\le i\le N$ and $\alpha\ge D_r\eta s^{\ell -1}h/N$.  In order to prove $Q_r$ we must prove that
\begin{equation}\label{eq:qis}
\pr{|X_{\ell}(B_i)|\, >\, \alpha}\, \le\, N^{O_k(1)}\, \exp\left(\frac{-\Omega_k(1) \alpha^{2/r}}{i\Delta_{r+1}^{2/r}}\right)\, .
\end{equation}

We recall that $X_{\ell}(B_i)=A_{\ell}(B_i)-\Ex{A_{\ell}(B_i)|B_{i-1}}$.  Our proof of~\eqref{eq:qis} is based on the following proposition on the deviation of $A_{\ell}(B_i)$ from its mean.  We set
\[
\lambda_{\ell}(i)\, :=\, \Ex{A_{\ell}(B_i)}\, =\, \frac{\ell  \binom{k}{\ell} h (i-1)_{\ell -1}}{(N)_{\ell}}\, .
\]
Note that $\lambda_{\ell}(B_i)$ is also equal to $L_{\ell}(i)-L_{\ell}(i-1)$.

\begin{prop}\label{prop:Q}
\[
\pr{\big|A_{\ell}(B_i)\, -\, \lambda_{\ell}(i)\big|\, >\, \alpha}\, \le \, N^{O_k(1)}\exp\left(\frac{-\Omega_k(1)\, \alpha^{2/r}}{i\Delta_{r+1}^{2/r}}\right)\, 
\]
for all $\alpha \ge D_r\eta s^{\ell -1}h/2N$.
\end{prop}

Let us first observe that~\eqref{eq:qis} follows from Proposition~\ref{prop:Q}.  Since $X_{\ell}(B_i)=A_{\ell}(B_i)-\Ex{A_{\ell}(B_i)|B_{i-1}}$, the event $|X_{\ell}(B_i)|\, >\, \alpha$ of~\eqref{eq:qis} may only occur if either $|A_{\ell}(B_i) -\, \lambda_{\ell}(i)|>\alpha/2$ or $|\Ex{A_{\ell}(B_i)|B_{i-1}} -\, \lambda_{\ell}(i)|>\alpha/2$, by the triangle inequality.  The first of these events has probability at most 
\[
N^{O_k(1)}\exp\left(\frac{-\Omega_k(1)\, \alpha^{2/r}}{i\Delta_{r+1}^{2/r}}\right)
\]
by Proposition~\ref{prop:Q}.  The second may only occur if there exists $x\in [N]\setminus B_{i-1}$ such that $|A_{\ell}(B_{i-1}\cup \{x\}) -\, \lambda_{\ell}(i)|>\alpha/2$ and so the same bound holds by Lemma~\ref{lem:pN} and Proposition~\ref{prop:Q}.  This completes the proof of~\eqref{eq:qis}.

All that remains is to prove Proposition~\ref{prop:Q}.  We shall base the proof of Proposition~\ref{prop:Q} on Lemma~\ref{lem:claim}, which shows how we may view $A_{\ell}$ in terms a deviation in a hypergraph $\HH(x)$, and Lemma~\ref{lem:fromP} which uses $P_r$ to bound the probability of such deviations.

It will be useful to condition on the element $b_i=x$ that is added as the $i$th element of the process.  Given that $b_i=x$ the set $B_i$ is distributed as
\[
B_i\, =\, B_{i-1}^{(x)}\, \cup \{x\}
\]
where $B_{i-1}^{(x)}$ is a uniformly random subset of $i-1$ elements of $[N]\setminus \{x\}$.  We also recall that we define the hypergraph
\[
\HH(x)\, :=\, \{f\setminus \{x\}\, :\, f\in E(\HH)\, , \, x\in f\}\, .
\]
The first lemma shows that $A_{\ell}(B^{(x)}_{i-1}\cup\{x\})$ may be expressed precisely in terms of the deviation of $(\ell-1)$ sets in the hypergraph $\HH(x)$

\begin{lemma}\label{lem:claim}  For each $x\in [N]$ 
\begin{align*}
A_{\ell}(B^{(x)}_{i-1}\cup\{x\})\, &=\, D_{\ell-1}^{\HH(x)}(B^{(x)}_{i-1}) \, +\,  \Ex{N_{\ell-1}^{\HH(x)}(B^{(x)}_{i-1})}\\
& =\, D_{\ell-1}^{\HH(x)}(B^{(x)}_{i-1}) \, +\, \lambda_{\ell}(i)\, \pm \, \frac{\eta\ell\binom{k}{\ell}h s^{\ell-1}}{N} \, .
\end{align*}
\end{lemma}

We use $P_r$ to provide a bound on the probability that $D_{\ell-1}^{\HH(x)}(B^{(x)}_{i-1})$ is large.

\begin{lemma}\label{lem:fromP} Let $r+1\le \ell\le k$.  Then
\[
\pr{\big|D_{\ell-1}^{\HH(x)}(B^{(x)}_{i-1})\big|\, >\, \alpha}\, \le \, N^{O_k(1)} \exp\left(\frac{-\Omega_k(1)\, \alpha^{2/r}}{i\Delta_{r+1}^{2/r}}\right)\,  
\]
for all $\alpha \ge C_r (3\eta)^{r/(r-1)}e(\HH(x))s^{(\ell-2)r/(r-1)}$.  In particular, the result holds for all $\alpha\ge D_r\eta s^{\ell -1}h/4 N$. 
\end{lemma}

Let us see how Proposition~\ref{prop:Q} follows from these lemmas.

\begin{proof}[Proof of Proposition~\ref{prop:Q}] By Lemma~\ref{lem:claim} we have that
\[
\big|A_{\ell}(B_i)\, -\, \lambda_{\ell}(i)\big|\, \le \, \big|D_{\ell-1}^{\HH(x)}(B^{(x)}_{i-1})\big|\, +\,  \frac{\eta\ell\binom{k}{\ell}h s^{\ell-1}}{N} \, .
\]
Since the second term on the right is at most $\alpha/2$ (this follows from the condition on $\alpha$ and the fact that $D_r\ge 4\ell\binom{k}{l}$), the event that $\big|A_{\ell}(B_i)\, -\, \lambda_{\ell}(i)\big|> \alpha$ is contained in the event that $|D_{\ell-1}^{\HH(x)}(B^{(x)}_{i-1})|>\alpha/2$.  The required bound now follows immediately from Lemma~\ref{lem:fromP}.
\end{proof}

All that remains is to prove Lemma~\ref{lem:claim} and Lemma~\ref{lem:fromP}.  We begin with Lemma~\ref{lem:claim}.

\begin{proof}[Proof of Lemma~\ref{lem:claim}] Recall that 
\[
A_{\ell}(B^{(x)}_{i-1}\cup\{x\})\, :=\, N_{\ell}(B^{(x)}_{i-1}\cup\{x\})\, -\, N_{\ell}(B^{(x)}_{i-1})
\]
which is precisely the number of pairs $(S,f)$ where $S\subseteq B^{(x)}_{i-1}$ is a subset of $\ell - 1$ elements, and $f$ is an edge of $\HH$ such that $S \cup \{x\} \subseteq f$.  Setting $f^{-}=f\setminus \{x\}$ we note that this condition is equivalent to the fact that $S \subseteq f^{-}$.  It follows that $A_{\ell}(B^{(x)}_{i-1}\cup\{x\})$ is precisely $N_{\ell-1}^{\HH(x)}(B^{(x)}_{i-1})$, and so
\[
A_{\ell}(B^{(x)}_{i-1}\cup\{x\})\, =\, D_{\ell-1}^{\HH(x)}(B^{(x)}_{i-1}) \, +\,  \Ex{N_{\ell-1}^{\HH(x)}(B^{(x)}_{i-1})}
\]
by the definition of $D_{\ell-1}^{\HH(x)}(B^{(x)}_{i-1})$ as the deviation of $N_{\ell-1}^{\HH(x)}(B^{(x)}_{i-1})$ from its mean.  All that remains is to prove that 
\[
\Ex{N_{\ell-1}^{\HH(x)}(B^{(x)}_{i-1})}\, =\, \lambda_{\ell}(i)\, \pm \, \frac{\eta\ell\binom{k}{\ell}s^{\ell-1}h}{N} \, .
\]
Since $\HH$ is $(r,\eta)$-near-regular (and so $(1,\eta)$-near-regular by Lemma~\ref{lem:nearreg}) we have $e(\HH(x))=d_{\HH}(x)=(1 \pm \eta) hk/N$.  We may also observe that $\HH(x)$ is $(k-1)$-uniform on $N-1$ vertices and so
\begin{align*}
\Ex{N_{\ell-1}^{\HH(x)}(B^{(x)}_{i-1})}\, &=\, \frac{e(\HH(x))\binom{k-1}{\ell-1} (i-1)_{\ell-1}}{(N-1)_{\ell-1}}\\
&=\, \frac{\ell\binom{k}{\ell}h (i-1)_{\ell-1}}{(N)_{\ell}} \, \pm \, \frac{\eta\ell\binom{k}{\ell}h(i-1)_{\ell-1}}{(N)_{\ell}} \\
&=\,  \lambda_{\ell}(i)\, \pm  \, \frac{\eta\ell\binom{k}{\ell}s^{\ell-1}h}{N} \, ,
\end{align*}
as required.
\end{proof}

\begin{proof}[Proof of Lemma~\ref{lem:fromP}]
We prove the required bound by applying the inequality given by $P_r$ to the hypergraph $\HH(x)$.  We observe that $\HH(x)$ is a $(k-1)$-uniform hypergraph on $N-1$ vertices.  We may also observe that $\HH(x)$ inherits the regularity condition $(r-1,3\eta)$-near-regular from $\HH$ by Lemma~\ref{lem:nearreg} and the maximum $r$-degree of $\HH(x)$ is at most $\Delta_{r+1}$.  By $P_r$ we have that
\[
\pr{\big|D_{\ell-1}^{\HH(x)}(B^{-}_{i-1})\big|\, >\, \alpha}\, \le \, N^{O_k(1)} \exp\left(\frac{-\Omega_k(1)\, \alpha^{2/r}}{i\Delta_{r+1}^{2/r}}\right)\,  
\]
for all $\alpha\ge C_r (3\eta)^{r/(r-1)}e(\HH(x))s^{(\ell-2)r/(r-1)}$.  This is exactly the result we need.  All that remains is to verify that this includes all $\alpha\ge D_r\eta s^{\ell -1}h/4 N$. 

We have that $D_r\ge 10k \cdot 3^{r/(r-1)}C_r$, $\eta<1$, $(\ell-2)r\le (\ell-1)(r-1)$ and $e(\HH(x))\le (1+\eta)hk/N\le 2hk/N$.  It follows that
\begin{align*}
\frac{D_r\eta h s^{\ell -1}}{4N}\, &\ge \, \frac{10C_r (3\eta)^{r/(r-1)}hk s^{(\ell -2)r/(r-1)}}{4N}\\
&\ge \, C_r(3\eta)^{r/(r-1)}e(\HH(x))s^{(\ell -2)r/(r-1)}\, .
\end{align*}
This confirms that the inequality holds across the whole of the range we claimed.
\end{proof}

\subsection{Q implies P}\label{sec:QtoP}

In this section, we will prove that $Q_r\, \Rightarrow\, P_{r+1}$. The main idea of the induction step is using our information about the magnitude of increments combined with Lemma ~\ref{lem:HAv}. Let us fix $k \geq j \geq r+1$, $\eta \in [0,3^{-r}]$ and a $k$-uniform $(r,\eta)$-near-regular hypergraph $\HH$ on $[N]$ with maximum $(r+1)$-degree $\Delta_{r+1}$. We recall the martingale representation
\[
D_j(B_m) = \sum_{i=1}^m Y_i,
\]
where
\[
Y_i\, =\, \sum_{\ell=1}^{j} \frac{(N-m)_{\ell}(m-i)_{j-\ell}}{(N-i)_j}\, \binom{k-\ell}{k-j}\, \, X_{\ell}(B_i)\, .
\]
We prove now an auxiliary lemma that controls the probability that the increments are large.
\begin{lemma} \label{lem:probinc}
If $\alpha \geq j k! D_r \eta t^{j-1}h/N$, then
\[
\pr{|Y_i|\, >\, \alpha}\, \le \, N^{O_k(1)}\exp\left(\frac{-\Omega_k(1) \alpha^{2/r}}{m\Delta_{r+1}^{2/r}}\right) .
\]
\end{lemma}

\begin{proof}
We first observe that
\begin{equation}\label{eq:boundYX}
|Y_i| \, \leq \, \sum_{\ell = 1}^{j} k! t^{j-\ell} |X_{\ell}(B_i)|,
\end{equation}
since
\[
\frac{(N-m)_{\ell}(m-i)_{j-\ell}}{(N-i)_j}\, \binom{k-\ell}{k-j}\, \leq \, t^{j-l} k!.
\]
By an application of Lemma ~\ref{lem:deterministic} and $Q_r$, since $\alpha \geq j k! D_r \eta t^{j-1}h/N$ the following bound holds for all $1 \leq \ell \leq k$:

\begin{equation}\label{eq:probX}
\pr{|X_{\ell}(B_i)|\, >\, \frac{\alpha t^{\ell-j}}{jk!} }\, \le\, N^{O_k(1)}\, \exp\left(\frac{-\Omega_k(1) \alpha^{2/r}}{m\Delta_{r+1}^{2/r}}\right).
\end{equation}

Finally, using \eqref{eq:boundYX}, \eqref{eq:probX} and the union bound, we have
\[
\pr{|Y_i|\, >\, \alpha}\, \le \, N^{O_k(1)}\exp\left(\frac{-\Omega_k(1) \alpha^{2/r}}{m\Delta_{r+1}^{2/r}}\right) .
\]
\end{proof}

We are now ready to prove $P_{r+1}$. Let $a\ge C_{r+1} \eta^{(r+1)/r}ht^{(j-1)(r+1)/r}$. Choosing $\alpha = a^{r/(r+1)}\Delta_{r+1}^{1/(r+1)}$, we can easily verify that
\[
\alpha \, \geq \, j k! D_r \eta t^{j-1}h/N,
\]
using $\Delta_{r+1} \geq h/(N^{r+1})$ and $C_{r+1} \geq k^2 (k!)^2 D_r^2$. By an application of the Azuma--Hoeffding inequality (the version given in Lemma ~\ref{lem:HAv}) with $c_i = \alpha$ for every $i$, we have
\[
\pr{D_j(B_m) > a} \, \leq \, \exp \left( \frac{-a^2}{2m \alpha^2} \right) + N \sum_{i=1}^m \pr{|Y_i| > \alpha} .
\]
By Lemma ~\ref{lem:probinc}, we obtain
\[
\pr{D_j(B_m) > a} \, \leq \, \exp \left( \frac{-a^2}{2m \alpha^2} \right) + N^{O_k(1)}\exp\left(\frac{-\Omega_k(1) \alpha^{2/r}}{m\Delta_{r+1}^{2/r}}\right).
\]
Since $\alpha = a^{r/(r+1)}\Delta_{r+1}^{1/(r+1)}$, this last inequality gives us
\[
\pr{D_j(B_m) > a} \, \leq \, N^{O_k(1)}\exp\left(\frac{-\Omega_k(1)\, a^{2/(r+1)}}{m\Delta_{r+1}^{2/(r+1)}}\right),
\]
which establishes the induction step.

Now that we have established the base case $P_1$ and the implications $P_r\Rightarrow Q_r$ and $Q_r\Rightarrow P_{r+1}$ for all $r\ge 1$ we have completed the proof of Theorem~\ref{thm:nearreg}.

\subsection{A weaker condition if $\HH$ is regular and $(r-1,\eta)$-near-regular}\label{sec:regand}

We remarked (Remark~\ref{rem:regand}) in the Introduction that the condition on $a$ may be weakened to
\[
a\, \ge\, \left( 10k! \right)^{10^r}e(\HH)\left(\frac{\eta m^{k-2} }{N^{k-2}}\right)^{r/(r-2)}
\]
if $\HH$ is regular and $(r-1,\eta)$-near-regular, for $r\ge 3$.  

The proof is essentially identical to that given above.  We highlight only the differences.  

Let $\HH$ be regular and $(r,\eta)$-near-regular.  In this case the $\pm\frac{\eta\ell\binom{k}{\ell}s^{\ell-1}h}{N}$ term in Lemma~\ref{lem:claim} is not necessary.  It follows that in this case $Q_r$ holds for all $\alpha\ge C_r(3\eta)^{r/(r-1)} h s^{(\ell-2)r/(r-1)}$ (by a simple adaptation of the proof of Proposition~\ref{prop:Q}).

It is then possible to prove the result of Lemma~\ref{lem:probinc} for all $\alpha\ge 10 j(k+1)! C_r\eta^{r/(r-1)} ht^{(j-2)r/(r-1)}/N$.  Following the rest of the proof of Section~\ref{sec:QtoP} we obtain the required result for all $a\, \ge\, C_{r+1} \eta^{(r+1)/(r-1)}h t^{(j-2)(r+1)/(r-1)}$.  Taking $j=k$ and swapping $r$ for $r-1$ we obtain the claimed result, which we now state as a proposition.

\begin{prop}\label{prop:regand}  Let $1\le r\le k$ and let $\eta\in [0,3^{-r+1}]$.
Let $\HH$ be a $k$-uniform hypergraph on $[N]$.  Suppose that $\HH$ is regular and $(r-1,\eta)$-near-regular with maximum $r$-degree $\Delta_r$.  Then
\[
\pr{|D^{\HH}(B_m)|>a}\, \le\, N^{O_k(1)}\, \exp\left(\frac{-\Omega_k(1) a^{2/r}}{m\Delta_{r}^{2/r}}\right)
\]
for all
\[
a\, \ge\, \left( 10k! \right)^{10^r}e(\HH)\left(\frac{\eta m^{k-2} }{N^{k-2}}\right)^{r/(r-2)}
\]
\end{prop}


\section{Deviations $D^{\HH}(B_p)$ -- Proof of Theorem~\ref{thm:pworld}}\label{sec:pworld}

Deviations $D(B_p)$ in the $p$-model are intimately related to deviations $D(B_m)$ in the $m$-model, via the identity
\eq{ptom}
\pr{D^{\HH}(B_p)\, >\, a}\, =\, \sum_{m=0}^{N}b_{N,p}(m) \, \pr{N^{\HH}(B_m)\, >\, p^{k}h\, +\, a}\, ,
\eqe
where $b_{N,p}(m):=\pr{\Bin(N,p)=m}$.  Recall that $L^{\HH}(p)=p^k h$ is the expected value of $N^{\HH}(B_p)$ and we study the probability that the deviation satisfies $D^{\HH}(B_p)>\delta_N L^{\HH}(p)$.

Our proof of Theorem~\ref{thm:pworld} consists of a lower bound (see Section~\ref{sec:LBpworld}) and an upper bound (see Section~\ref{sec:UBpworld}).  The lower bound is based on  on the fact that $\pr{N(B_m)\, >\, p^{k}h\, +\, a}$ is increasing in $m$ and so
\eq{forLB}
\pr{D(B_p)\, >\, \delta_N L(p)}\, \ge \, B_{N,p}(m_{+}) \, \pr{N(B_{m_+})\, >\, \big(1 \, +\, \delta_N\big) L(p)}
\eqe
for all $m_+ \ge 0$, where $B_{N,p}(m):=\pr{\Bin(N,p)\ge m}$.  For a particular choice of $m_{+}=m_{+}(\delta_N)$ we shall prove that 
\[
B_{N,p}(m_{+}) \, \pr{N(B_{m_+})\, >\, \big(1 \, +\, \delta_N \big) L(p)}\, =\, \exp\left( -(1+o(1))\frac{\delta_N^2 pN}{2k^2(1-p)}\right)\, 
\]
which gives the required lower bound.  

We need to work harder to prove the upper bound.  We must control all contributions to the sum~\eqr{ptom}.  We again use that $\pr{N(B_m)\, >\, p^{k}h\, +\, a}$ is increasing to observe that
\[
\pr{D(B_p)\, >\, \delta_N L(p)}\, \le \, \pr{N(B_{m_-})\, >\, \big(1 \, +\, \delta_N\big) L(p)}\, +\, B_{N,p}(m_-)
\]
for all $m_-\ge 0$.  We shall then choose $m_{-}=m_{-}(\delta_N)$ such that 
\[
B_{N,p}(m_{-})\, =\, \exp\left( -(1+o(1))\frac{\delta_N^2 pN}{2k^2(1-p)}\right)
\]
and
\[
\pr{N(B_{m_-})\, >\, \big(1 \, +\, \delta_N\big) L(p)}\, \ll\, \exp\left(\frac{-\delta_N^2 pN}{2k^2(1-p)}\right)\, .
\]
The latter inequality is proved using Theorem~\ref{thm:nearreg}.

Based on the above sketch it is clear that the probabilities $b_{N,p}(m)$ and $B_{N,p}(m)$, related to the binomial distribution, are central to our proof.  While more precise estimates, up to a multiplicative factor of $1+o(1)$ are known\footnote{see for example the bound in~\cite{GGS2019}, adapted from Bahadur~\cite{B1960}}, the following is sufficient for our purposes.  Throughout the section we set $q:=1-p$.

\begin{theorem}\label{thm:binom}
Suppose that $(x_N)$ is a sequence such that $1 \ll x_N \ll \sqrt{Npq} $. Then
$$
b_{N,p}(\lfloor pN + x_N\sqrt{Npq} \rfloor) \, = \, \frac{1}{\sqrt{Npq}} \exp \left( -(1+o(1))\frac{x_N^2}{2}\right)
$$
and
$$
B_{N,p}(\lfloor pN+x_N\sqrt{Npq}\rfloor) \, =\, \exp \left( -(1+o(1)) \frac{x_N^2}{2} \right).
$$
\end{theorem}

To see that these bounds do indeed follow from Theorem 1.13 of~\cite{GGS2019} (for example) simply note that for any sequence $1\ll x_N\ll \sqrt{Npq}$ we have that:
\begin{enumerate}
\item[(i)] the $E(x_N,N)$ expression is at most $o(x_N^2)$ (by comparison with a geometric series), and 
\item[(ii)] the multiplicative term satisfies $1/\sqrt{2\pi x_N}=\exp(-\log(2\pi x_N)/2)=\exp(o(x_N^2))$\, .
\end{enumerate}

Both of the values $m_{-}$ and $m_{+}$ discussed above will be chosen in relation to
$$
m_{*}\,  :=\, (1+\delta_{N})^{1/k} pN\, ,
$$
which is chosen so that $L(m^*)=L(p)$.  Let us also define
$$
x(m) := \frac{m-pN}{\sqrt{Npq}}
$$
in general and, in particular, set $x_{*} := x(m_{*})$.

\subsection{Lower Bound} \label{sec:LBpworld}
First we choose a sequence $f_N$ such that
\[
\max \left \{ \eta^{r/(r-1)} p^{(k-1)/(r-1)}N,\frac{p^{r/2-k+1}N^{r/2+1}\Delta_r (\log N)^{r/2}}{h} \right \} \ll f_N \ll \delta_N pN.
\]
Note that since 
\[
\delta_N \gg \max \left\{\frac{\Delta_r (N\log N)^{r/2}}{p^{k-r/2}h},\big(\eta^{r} p^{k-r}\big)^{1/(r-1)} \right\}
\]
we can take such sequence. We then choose $m_+ = m_* + f_N$ and we also set $x_+ = x(m_+)$.

Now we prove the following two lemmas, which together with \eqr{forLB}, will give us the desired lower bound.

\begin{lemma}\label{lem:probm+}
\[
\pr{N(B_{m_+}) > (1+\delta_N) L(p)} = 1 - o(1).
\]
\end{lemma}

\begin{proof}
Note first that it suffices to prove that
\[
\pr{D(B_{m_+}) \leq (1+\delta_N)L(p) - L(m_+)}\, =\, o(1).
\]
Observe now that
\begin{equation}
\begin{split}
(1+\delta_N)L(p) - L(m_+) &= \, hp^k \left[ (1+\delta_N)p^k  - \frac{(m_+)_k}{(N)_k} \right] \nonumber \\
&\leq  \, h \left[ (1+\delta_N)p^k - \frac{(m_* + f_N-k)^k}{N^k} \right]\nonumber \\
& \leq \, h \left[ (1+\delta_N)p^k - \frac{m_*^k}{N^k} - \frac{m_* ^{k-1}f_N}{N^k} \right]\nonumber \\
& \leq \, -\frac{hp^{k-1}f_N}{N}. \label{eq:change}
\end{split}
\end{equation}
So we have
\[
\pr{D(B_{m_+})  \leq  (1+\delta_N)L(p) - L(m_+)} \le \pr{D(B_{m_+}) \, \leq \, -\frac{hp^{k-1}f_N}{N}}.
\]
Let us denote $a = (hp^{k-1}f_N)/N$. Since $f_N \gg \eta^{r/(r-1)} p^{(k-1)/(r-1)}N$, if $N$ is sufficiently large, we can apply Theorem ~\ref{thm:nearreg} and obtain
\[
\pr{D(B_{m_+}) \leq -a} \, \le \, N^{O_k(1)}\, \exp\left(\frac{-\Omega_k(1) a^{2/r}}{m\Delta_{r}^{2/r}}\right).
\]
Using now that
\[
f_N\, \gg\, \frac{p^{r/2-k+1}N^{r/2+1}\Delta_r (\log N)^{r/2}}{h}
\]
we can easily verify that
\[
N^{O_k(1)}\, \exp\left(\frac{-\Omega_k(1) a^{2/r}}{m\Delta_{r}^{2/r}}\right) = o(1),
\]
which gives us the desired result.
\end{proof}


\begin{lemma}\label{lem:mplus}
\[
B_{N,p}(m_+) \, =\,  \exp\left( -(1+o(1))\frac{\delta_N^2 pN}{2k^2(1-p)}\right).
\]
\end{lemma}

\begin{proof}
By Theorem ~\ref{thm:binom}, we have
\[
B_{N,p}(m_+) \, = \, \exp\left( -(1+o(1))\frac{x_{+}^2}{2}\right).
\]
Observe now that
\[
\frac{x_+}{x_*} \, =\, 1 + \frac{f_N}{pN[(1+\delta_N)^{1/k}-1]}.
\]

Since $\delta_N \ll 1$, $(1+\delta_N)^{1/k} - 1 = \Theta(\delta_N)$ and as $f_N \ll \delta_N pN$, we obtain $x_+ = (1+o(1))x_*$, which gives us
\[
B_{N,p}(m_+) \, = \, \exp\left( -(1+o(1))\frac{x_{*}^2}{2}\right).
\]
Finally, note that
\[
x_* \, = \, \sqrt{\frac{pN}{1-p}} \left( (1+\delta_N)^{1/k}-1 \right)
\]
and since $(1+\delta_N)^{1/k}-1 = (1+o(1)) \delta_N/k$, we obtain the required result.
\end{proof}

\subsection{Upper Bound} \label{sec:UBpworld}

Our upper bound on $\pr{D^{\HH} (B_p) > \delta_N L^{\HH}(p)}$ requires us to control all the terms of the sum~\eqr{ptom}, i.e., all the terms in the sum
\[
\pr{D^{\HH}(B_p)\, >\, \delta_N L^{\HH}(p)}\, =\, \sum_{m=0}^{N}b_{N,p}(m) \, \pr{N^{\HH}(B_m)\, >\, p^{k}h\, +\, \delta_N L^{\HH}(p)}\, .
\]
In fact we do not require a very precise analysis.  We shall simply break the sum into two parts $m< m_{-}$ and $m\ge m_{-}$ for a value of $m_-$ we define below.  We bound the terms $m\ge m_{-}$ using only the first probability (the binomial) and the terms $m<m_{-}$ using only the second probability (the deviation in the model $B_m$).

Let $g_N$ be a sequence satisfying
\[
\frac{\eta^{r/(r-1)}p^{(k-1)/(r-1)}N}{h}\, ,\, \frac{N^{r/2+1}(\log{N})^{r/2}\Delta_r}{p^{k-r/2-1}h}\, , \, \frac{\delta_N^r N^{r+1}\Delta_r}{p^{k-r-1}h} \ll\, g_N\, \ll\, \delta_N pN\, .
\]
It is certainly possible to choose such a sequence by the conditions on $\delta_N$ in Theorem~\ref{thm:pworld}.

We define $m_-:=m_*-g_N$ and set $x_{-}=x(m_{-})$.  By a calculation similar to that given in~\eqr{change} we have
\eq{mminus}
L(m_-)\, \le\, (1+\delta_N)L(p)\,  -\, \frac{hp^{k-1}g_N}{N}\, .
\eqe

We now bound the two parts of the sum.  First, for the part $m\ge m_{-}$ we simply use that the sum of these terms is at most
\[
B_{N,p}(m_{-})\, =\, \exp\left( -(1+o(1))\frac{\delta_N^2 pN}{2k^2(1-p)}\right)\, .
\]
This may be verified by simply following the proof of Lemma~\ref{lem:mplus} and using that $g_N\ll \delta_N pN$.

Now, we bound the rest of the sum by $\exp(-x_{*}^2)$ using the following lemma.  This will complete the proof of the upper bound.

\begin{lemma}
\[
\sum_{m=0}^{m_-} \pr{N^{\HH}(B_m)\, >\, p^{k}h\, +\, \delta_N L^{\HH}_(p)}\, \le\, \exp(-x_{*}^{2})\, .
\]
\end{lemma}

\begin{proof}
Since $\pr{N^{\HH}(B_m) > p^{k}h + \delta_N L^{\HH}(p)}$ is increasing in $m$ it suffices to prove that
\eq{need}
\pr{N^{\HH}(B_{m_-})\, >\, p^{k}h\, +\, \delta_N L^{\HH}(p)}\, \le\, \frac{\exp(-x_{*}^{2})}{N}\, .
\eqe
By~\eqr{mminus} this event is contained in the event that $D^{\HH}(B_{m_{-}}) > hp^{k-1}g_N/ N$.  Three lower bounds on $g_N$ were given above.  The first ensures that we may apply Theorem~\ref{thm:nearreg} to bound the probability of the deviation $D^{\HH}(B_{m_{-}})\, >\, hp^{k-1}g_N/ N$.  The second and third give that the resulting bound is at most $ N^{O_k(1)}\exp(-\omega(\log{N}))$ and $ N^{O_k(1)}\exp(\omega(x_{*}^2))$ respectively.  In particular, for any constant $C$ we have that
\[
\pr{D^{\HH}(B_{m_{-}})\, >\, hp^{k-1}g_N/ N}\, \le\, N^{O_k(1)} \exp(-C\log{N}-x_{*}^2)
\]
for all sufficiently large $N$.  Choosing $C$ to be one larger than the constant of the $O_k(1)$ we obtain~\eqr{need}, and so complete the proof of the lemma.
\end{proof}

\section{Arithmetic configurations in random sets}\label{sec:applications}

In this section we collect some applications of Theorems  \ref{thm:nearreg} and \ref{thm:pworld} to illustrate its use in obtaining bounds of deviations for the count of arithmetic structures in random sets. To simplify matters we will consider the ambient group to be the cyclic group $\Z/N\Z$ with $N$ prime.

\subsection{$k$--progressions}

The initial motivation of this article was to analyze moderate deviations in counting the number of $k$--term arithmetic progressions ($k$--progressions for short) in random sets. As it has been mentioned in the Introduction, large deviations for this problem have been intensively studied and they are currently well understood, see \cite{JR2011, W2017, BGSZ2016, HMS2019, BG2020}. It was also proved by Berkowitz, Sah and Sawhney~\cite{BSS2020} that the number of $k$--progressions in a dense binomial random subset of $\ZZ_n$ does not obey a local central limit theorem, despite obeying a central limit theorem. We consider here moderate deviations for the $m$--model first.  Let $N^k(B_m)$ denote the number of $k$--progressions in a random subset $B_m$ of the cyclic group $\Z/N\Z$, $N$ prime, and by $D^k(B_m)$ the deviation of $N^k(B_m)$. 

Let  $\HH_k$ be the $k$--uniform hypergraph with vertex set $\Z/N\Z$ and edges the nontrivial $k$--progressions $\{x,x+d,\ldots ,x+(k-1)d\}$ with $x\in \Z / N\Z$ and $d\in \{1,2,\ldots ,(N-1)/2\}$. Every pair $\{x,y\}$ of elements in $\Z/N\Z$ is contained in $\binom{k}{2}$ edges of $\HH_k$, so that the hypergraph is $2$--regular and hence $1$--regular.  A direct application of Theorem \ref{thm:reg} for $r=2$ gives the following result.

\begin{theorem}\label{thm:kprog} For each $a>0$,
$$
\pr{|D^k(B_m)|>a}\le N^{c_1}\exp \left(-c_2\frac{a}{m}\right)
$$
for some constants $c_1, c_2$ which depend only on $k$.
\end{theorem}


The expected value of $N^k(B_m)$ in this example is  (see \eqref{eq:Lj})
$$
L^k(m)={N\choose 2}\frac{(m)_k}{(N)_k}\sim\frac{m^k}{2N^{k-2}}.
$$
Bounds on moderate deviations are thus obtained in Theorem \ref{thm:kprog} for $a\ll m^k/N^{k-2}$. 

In particular, for $m\log N\ll a \ll m^k/N^{k-2}$ we obtain exponentially small bounds for moderate deviations of the $k$--progressions count, which apply to random sets of  $\Z/N\Z$ of size $m\gg (N^{k-2}\log N)^{1/(k-1)}$. This is slightly above the threshold cardinality for the existence of $k$--progressions in a random set (see R\"odl, Ruci\'nski  \cite{RR1997}).  One can not expect to obtain such small deviation for smaller sets since, as shown in Ru\'e, Spiegel and Zumalac\'arregui  \cite{RSZ2018}, the count of $k$--progressions within the threshold window converges to a Poisson distribution. 

We note that the term in the exponential in Theorem \ref{thm:reg} for $r=3$ corresponding to  this example would be $-\Omega_k(1)a^{2/3}/m$ (as $\Delta_3=O_k(1)$), while for $r=1$ we have $\Delta_1\sim N/2$ and  we would obtain  $-\Omega_k(1)a^2/mN$. Both are worse than the one obtained for $r=2$ as displayed in the bound of Theorem \ref{thm:3apm}. This exemplifies  Remark \ref{rem:minr} in the Introduction.

By following the proof of Theorem \ref{thm:nearreg} in the case  $k=3$  one can make the $o_k(1)$ and $\Omega_k(1)$ terms in the bound of the Theorem explicit (even if not optimized).  We include this proof here as it is simpler than the general one described in Section \ref{sec:mproof} and may illustrate its main lines.

\begin{theorem}\label{thm:3apm} Let $D^3(B_m)$ denote the deviation in counting $3$--progressions in $\Z/N\Z$, with $N$ prime. For every $a>0$, we have
$$
\pr{|D^3(B_m)|>a}\le (Nm+1)\,\exp \left(-\frac{a}{9m}\right).
$$
\end{theorem}

\begin{proof} In this case the hypergraph $\HH_3$  defined in the proof of Theorem \ref{thm:kprog} is $2$--regular. Hence, we have $D_1(B_m)=D_2(B_m)=0$ (by Lemma~\ref{lem:deterministic}, with $\eta=0$) and the martingale representation \eqref{eq:Mart} reduces to 
$$
D_3^3(B_m)\, =\, \sum_{j=1}^{m}  \frac{(N-m)_{3}}{(N-j)_3}X_3(B_j).
$$
We bound the probability that $|X_3(B_j)|$ is large  as in the proof of Lemma \ref{lem:fromP} by considering the hypergraph $\HH(x)$ which is a $2$--uniform $1$--regular hypergraph with $\Delta_1(\HH_3(x))=\Delta_2(\HH_3)=3$. It follows from the Azuma--Hoeffding inequality (the argument in Section \ref{sec:base}) that, for each $\alpha>0$,
$$
\pr{X_3(B_j)}>\alpha)\, <\, \exp\left(-\frac{\alpha^2}{9j}\right).
$$
Therefore, as in the argument in Section \ref{sec:QtoP}, Lemma ~\ref{lem:HAv} gives, for each $a>0$ and $\alpha>0$,
$$
\pr{D_3^3(B_m)>a}\, \le\, \exp\left(-\frac{a^2}{2m\alpha^2}\right)+N\sum_{j=1}^m\exp\left(-\frac{\alpha^2}{9j}\right)\le \exp\left(-\frac{a^2}{2m\alpha^2}\right)+Nm\exp\left(-\frac{\alpha^2}{9m}\right).
$$
The result follows by choosing $\alpha=a^{1/2}$. \end{proof}

The bounds on deviations for the counting in the $m$--model can be transferred to the $p$-binomial model as described in Section \ref{sec:pworld}. As mentioned in the Introduction, Theorem \ref{thm:APpworld} follows from  Theorem \ref{thm:pworld} taking into account that the hypergraph $\HH_3$ of $3$--progressions in $\Z/N\Z$ is $1$--regular. We can analogously derive the following result for $k$--progressions from Theorem  \ref{thm:pworld} using again the $1$--regularity of the hypergraph  $\HH_k$ of $k$--progressions.

\begin{theorem}\label{thm:APkpworld}
Let  $D^{k}(B_p)$ denote the deviation of the $k$--progressions count in a $p$-random subset $B_p$ of $\Z/N\Z$, $N$ prime.

Let $\delta_N$ be a sequence satisfying
\[
\max \left \{\frac{\log{N}}{p^{k-1}N^2}\, ,\, \frac{1}{\sqrt{pN}} \right \}\, \ll \,\delta_N\, \ll\, p^{k-2}.\]
Then,
\[
\pr{D^{k} (B_p) > \delta_N L^k(p)} \, =\,  \exp \left( -(1+o(1))\frac{\delta_N^2 pN}{2k^2(1-p)}\right)\, .
\]
Furthermore, the same bounds apply to the corresponding negative deviations.
\end{theorem}

As the expected number of $k$--progressions in $B_p$ is $L^{k}(p)=p^k\binom{N}{2}$, Theorem \ref{thm:APkpworld} provides exponentially small bounds of moderate deviations   for $p\gg N^{-1/(2k-3)}$.

\subsection{Schur equation} In addition to the $3$--progressions,  let us mention another $3$--variable case, the Schur equation $x+y=z$, related to problems on  sum--free sets. Threshold cardinalities and large deviations for the Schur equation have also been addressed in the literature, see e.g. \cite{GRR1996,W2017}. Consider the hypergraph $\HH$ with vertex set the nonzero elements of $\Z/N\Z$ and edge set the nontrivial Schur triples $\{x,y,z\}$ with $x+y=z$, none of them zero. Every nonzero element $x$ in $\Z/N\Z$ belongs to $N-3$ triples of the form $\{x,y,x+y\}$ with $y\not\in \{ 0,x,-x\}$ and $(N-3)/2$ triples of the form $\{y,x-y,x\}$ with $y\not\in \{ 0,x,x/2\}$. Therefore, $\HH$   is $1$--regular with $d_1(\HH)=3(N-3)/2$ and  maximum $2$--degree $\Delta_2(\HH)=3$.  Let $D^3(B_m)$ denote the deviation on the count of Schur triples of nonzero elements in a random set $B_m$ of $(\Z/N\Z)^*$. A proof analogous to the one of  Theorem \ref{thm:3apm}  gives, for all $a>0$,
$$
\pr{D^3(B_m)>a}\, <\, N^2\exp\left(-c\left(\frac{a}{m}\right)\right),
$$
for some constant $c$. The above bound is analogous to the one in Theorem \ref{thm:3apm} and applies to the same ranges of $a$ and $m$. The bounds on the deviations can be analogously transferred to the $p$--binomial model to obtain a bound as in Theorem \ref{thm:APpworld} for the deviation of Schur triples in this model. 

This example exemplifies some minor adjustments which can be made to apply Theorem \ref{thm:reg} for $r=2$ to the count of solutions of general $3$--term linear equations of the form $ax+by+cz=d$ in $\Z/N\Z$. The corresponding hypergraphs may fail to be regular or near--regular by the degree of some pairs of elements, depending on the values of $a,b,c$, but can be regularized without affecting the counts in a substantial way. We discuss the general approach for linear systems in Section \ref{sec:ls} below.

\subsection{Sidon equation} The Sidon equation $x+y=z+t$  has also been treated in the literature, see e.g. \cite{K2013, RSZ2018}. In particular, bounds of the deviation on the number of solutions of the Sidon equation are given  in \cite[Lemma 5.3]{K2013} by means of the Kim--Vu polynomial concentration inequality \cite{KV2000}.  An analogous analysis can be carried over in our context to obtain bounds on moderate deviations giving more precise results in an appropriate range of the size of random sets. The number of nontrivial quadruples $(x,y,z,t)$ in a set $A\subset \Z/N\Z$ which satisfy the Sidon equation $x+y=z+t$ with $\{x,y\}\neq\{z,t\}$ is also called the Additive Energy of the set (which counts ordered quadruples).  For unordered subsets, solutions of the Sidon equation in $A$ are either $3$--progressions $\{x,y,z=(x+y)/2\}$, which satisfy $x+y=2z$, or quadruples $\{x,y,z,t\}$ of distinct elements. 

We consider the $4$--uniform hypergraph $\HH_4$ which has vertex set $\Z/N\Z$ (as usual we consider $N$ prime) and edges the Sidon quadruples $\{x,y,z,t\}$ with $x+y=z+t$. Every pair $\{x,y\}$ is contained in $(N-3)$ quadruples of the form $\{x,y,z,x+y-z\}$ with $z\not\in \{x,y,(x+y)/2\}$ and $(N-3)$ quadruples of the form $\{x,y,z,x+z-y\}$ with $z\not\in \{x,y,(x-y)/2\}$, so that $\HH_4$ is $2$--regular with $d_2(\HH_4)=2(N-3)$. Moreover,  
each triple  $\{x,y,z\}$ which is not in arithmetic progression belongs to three distinct quadruples $\{x,y,z,t\}$ with $t\in \{x+y-z,x+z-y,y+z-x\}$,  while triples which are $3$--progressions, say $z=(x+y)/2$,  belong to two distinct quadruples. Hence, and $\Delta_3(\HH_4)=3$.  By denoting by $D^4(B_m)$ the deviation in the count of Sidon quadruples of paiwise distinct elements,  Theorem \ref{thm:nearreg}  with $r=3$ gives, for all $a>0$,
$$
\pr{D^4(B_m)>a}\, \le\, N^{c_1}\exp\left(-c_2\frac{a^{2/3}}{m}\right),
$$
for some constants $c_1, c_2$. By bounding the deviation $D^S(B_m)$ on the count of the number of solutions of the Sidon equation by the sum of deviations on the count of quadruples of distinct elements and the count of $3$--progressions from Theorem \ref{thm:3apm} we obtain
$$
\Pr (D^S(B_m)>a)\, \le\, N^{c_1}\exp\left(-c_S\frac{a^{2/3}}{m}\right),
$$
for some constant $c_S$. The expected number  of solutions of the Sidon equation in a random set $B_m$ is 
$$
L^S(m)\, =\, \frac{(m)_3}{2(N-3)}+\frac{(m)_4}{2(N-4)}.
$$
Therefore, for $a\ll m^4/N$ and $a^{2/3}/m\gg \log N$ we obtain exponentially small bounds for the deviation  in the count of solutions of the Sidon equation, which apply to random sets of cardinality $m\gg N^{2/5}(\log N)^{3/5}$. This shows that, for these values of $m$,  the Additive Energy of a random $m$--set in $\Z/N\Z$ is highly concentrated on its mean value. We observe that in this case the threshold for the appearance of solutions of the Sidon equation is $N^{1/4}$ and our bound starts to be effective only above $N^{2/5}$.

As in the preceding examples, Theorem \ref{thm:pworld} can be applied to transfer the bounds to the $p$--binomial model. Let $D^4(B_p)$ denote the deviation on the number of solutions to the Sidon equation with pairwise distinct entries in a random set $B_p$ of $\Z/N\Z$ and $L^4(p)$ its mean value.  Taking into account that the hypergraph $\HH_4$ is $2$--regular and has constant maximum degree $\Delta_3(\HH_4)=3$, Theorem \ref{thm:pworld} with $r=3$ gives
$$
\Pr (D^4(B_p)> \delta_N L^4(p))\, =\, \exp\left(-(1+o(1))\frac{\delta^2_NpN}{32(1-p)}\right),
$$
for every sequence $\delta_N$ satisfying 
$$
\max \left\{ \frac{1}{p^{5/2}}\left( \frac{\log N}{N} \right)^{3/2} ,\frac{1}{\sqrt{pN}} \right\}\, \ll\, \delta_N \ll p^{1/2}.
$$
The above condition on $\delta_N$ gives a meaningful range of applications when $p\gg (\log N/N)^{1/2}$.

\subsection{Linear systems}\label{sec:ls} The arithmetic configurations considered in the above examples correspond to solutions of linear systems 
$$
Ax\, =\, 0
$$
in some finite field $\FF_N$, $N$ a prime,  where $A$ is a $(l\times k)$ matrix with integer entries, $l\le k-2$, $x\in \FF_N^k$. This general setting has been widely addressed in the literature, see e.g. \cite{RR1997, JR2011, RSZ2018}. Conditions are imposed  on the matrix $A$ to avoid some degenerate cases, the most natural one is that  all $l\times l$ submatrices of $A$ are nonsingular. Then  substitution of $k-l$ entries in $x$ gives a unique solution to the equation. It follows that the hypergraph $\HH_A$ with vertex set $\FF_N$ and edges the $k$--subsets which are entries of a solution to the linear system has  maximum $(k-l)$ degree $\Delta_{k-l}(\HH_A)\le {k\choose k-l}$. One can apply Theorem \ref{thm:nearreg} in this context whenever the hypergraph is $(k-l-1,\eta)$--near regular for an appropriate value of $\eta$.  

Substitution of $k-l-1$ entries in $x$ give rise to $N$ solutions to the system. If the equation $x_i=x_j$ is linearly independent with the rows of $A$ for every pair $i,j\in [k]$, then $O_k(1)$ solutions may have repeated entries.  Matrices $A$ which satisfy this last condition are called {\it irredundant} (see \cite{RR1997}), a property that we will also assume.  Moreover, $O_k(1)$ substitutions may give rise to the same solution. Therefore, under the above conditions on $A$, every $(k-l-1)$--set belongs to ${k\choose k-l-1}N-O_k(1)$ edges of the hypergraph. It follows that the hypergraph $\HH_A$ is $(k-l-1,\eta)$--near regular where $\eta = O_k(1/N)$.  Under the above stated conditions on $A$, let $D^{\HH_A}(B_m)$ denote  the deviation on the count of solutions to the system $Ax=0$  with pairwise distinct entries in a random set $B_m$. Hence Theorem \ref{thm:nearreg} with $r=k-l$ gives, for some constant $c_k$  and every sufficiently large $N$, 
$$
\pr{D^{\HH_A}(B_m)>a}\, \le\,  N^{O_k(1)}\exp\left(-\frac{\Omega_k(1)a^{2/(k-l)}}{m}\right)\, ,
$$
for all $a>c_k\left(\frac{m^{k-1}}{N^{l+1}}\right)^{(k-l)/(k-l-1)}$. 
Since the expected value of the number of solutions is of the order $m^k/N^l$, for
\[
\max\left\{ m^{(k-l)/2} (\log N)^{(k-l)/2}, \left( \frac{m^{k-1}}{N^{l+1}}\right)^{(k-l)/(k-l-1)} \right\} \ll a \ll \frac{m^k}{N^l},
\]
we obtain exponentially small bounds for moderate deviations, which apply to random sets of cardinality $m \gg N^{2l/(k+l)}(\log N)^{(k-l)/(k+l)}$.

Let $D^{\HH_A}(B_p)$ denote the deviation on the count of solutions to the system $Ax = 0$ with pairwise distinct entries in a random set $B_p$ of $\ZZ/N\ZZ$ and let $L^{\HH_A}(p)$ be its expected value. By using Theorem \ref{thm:pworld}  we obtain  the corresponding bound for moderate deviations in the $p$--model
$$
\pr{D^{\HH_{A}} (B_p)\, >\, \delta_N L^{\HH_{A}}(p)} \, =\,  \exp \left( -(1+o(1))\frac{\delta_N^2 pN}{2k^2(1-p)}\right)\, .
$$
for every sequence $\delta_N$ satisfying
\[
\max \left\{ \frac{1}{p^{(k+l)/2}}\left( \frac{\log N}{N}\right)^{(k-l)/2}  , \frac{1}{\sqrt{pN}} \right\} \, \ll\, \delta_N \, \ll\, p^{l/(k-l-1)}
\]
and so this is applicable when $p \gg (\log N/N)^{(k-l-1)/(k+l-1)}$.


\section{Lower Bounds}\label{sec:LB}

The aim of this section is to prove that the inequality established by Theorem~\ref{thm:nearreg} is tight up to a constant in the exponent, when $m = \Omega (N)$, for a family of hypergraphs which satisfy the conditions of the theorem.  The \emph{density} of a $k$-uniform hypergraph $\HH$ with $h$ edges is $h/\binom{N}{k}$.  We provide examples across a range of densities.

\begin{samepage}
\begin{prop}\label{prop:LB} Let $0 < \tau < 1/2$.  For each $r\ge 1$ and each sequence $\gamma_n$ satisfying $\Theta_r(1/n)\le \gamma_n\le \Theta_r(1)$ for infinitely many values of $N$, there exists an $N$-vertex $(r+1)$-uniform hypergraph $\HH$ with density $\Theta_r(\gamma_N)$, which is $(r-1,\eta)$-near-regular for $\eta=O_r(1/\gamma_N N)$ and with maximum $r$-degree $\Delta_r$, and for which
\[
\pr{D^{\HH}(B_m) > a}\, \geq\, N^{-O_{r,\tau}(1)} \exp \left( \frac{-O_{r,\tau}(1)a^{2/r}}{m \Delta_r^{2/r}} \right),
\]
for all $\max \{h/N, \eta h \} \le a\le \Omega_{r,\tau}(1)h$ and all $\tau N\le m\le (1-\tau)N$.
\end{prop}
\end{samepage}

In the cases $r=1$ and $r=2$ particularly simple constructions exist.  For $r=1$ one may take a regular graph on half the vertices and leave the remaining vertices isolated.  For $r=2$ one may take two disjoint copies of a regular $3$-uniform hypergraph.  Even for $r=3$ it is not difficult to describe a family of examples: Suppose $N=3s$ and partition the vertex set into $3$ equal parts $V=V_1\cup V_2\cup V_3$, a set $e$ of four vertices will be an edge of $\HH$ if it has $2$ vertices in each of two parts.

In general, for $r\ge 4$, we need to work harder to produce the examples.  Based on the $r=2$ and $r=3$ cases we see that it is useful to partition the vertex set.  Given a partition $V=V_1\cup \dots \cup V_{\ell}$ into $\ell$ equal parts of size $s:=N/\ell$ we will label the vertices $v_{i,j}:i\in [\ell],j\in [s]$.  We call $\HH$ an $\ell$-part hypergraph if it is invariant under permutations of the parts $V_i:i\in [\ell]$.  That is, given a permutation $\pi$ of $[\ell]$ then $\{v_{i_1,j_1},\dots ,v_{i_k,j_k}\}$ is an edge of $\HH$ if and only if $\{v_{\pi(i_1),j_1},\dots ,v_{\pi(i_k),j_k}\}$ is an edge of $\HH$.  For example, the $r=2$ case above is a $3$-uniform $2$-part hypergraph and the $r=3$ case is a $4$-uniform $3$-part hypergraph.  For general $r$ the hypergraphs we construct will be $(r+1)$-uniform and $\ell$-part for some $\ell$ chosen sufficiently large.  (In fact, we take $\ell=4(r+1)!$)

We consider the deviation event $D^{\HH}(B_m) > a$ relative to an auxiliary event in which the random set $B_m$ is unevenly distributed with respect to the parts $V_1,\dots ,V_{\ell}$.  Specifically, we define $E^{\ell}_{\eps}$ to be the event\footnote{the $\pm 1$ is only included as the values themselves may not be integers} that 
\begin{align*}
|B_m\cap V_1|\, & =\, (1+2\eps)\frac{m}{\ell}\, \pm\, 1\, .\\
|B_m\cap V_i|\, & =\, (1-\eps)\frac{m}{\ell}\, \pm\, 1\, \qquad\quad i\, =\, 2,3\, .\\
|B_m\cap V_i|\, & =\, \frac{m}{\ell}\, \pm\, 1\qquad\qquad\qquad\, i\, =\, 4,\dots,\ell\, .
\end{align*}
Since
\begin{equation}\label{eq:LB}
\pr{D^{\HH}(B_m) > a}\, \ge\, \pr{E^{\ell}_{\eps}}\, \pr{D^{\HH}(B_m) > a\, \big|\, E^{\ell}_{\eps}}
\end{equation}
it will suffice to prove lower bounds on $\pr{E^{\ell}_{\eps}}$ (see Lemma~\ref{lem:E}) and $\pr{D^{\HH}(B_m) > a| E^{\ell}_{\eps}}$ (see Lemma~\ref{lem:effect}).  

The responsiveness of the deviation $D^{\HH}(B_m)$ to the uneven distribution $E^{\ell}_{\eps}$ depends on the hypergraph $\HH$.  In particular, it depends on the coefficients of the following polynomial
$$
Q^{\HH}(x) = \sum_{e \in E(\HH)} (1+2x)^{e_1} (1-x)^{e_2 + e_3}
$$
where $e_i:=|e\cap V_i|$ for $i\in [\ell]$, which may be defined for any $\ell$-part hypergaph $\HH$.  Let $c_j^{\HH}$ be the coefficient of $x^j$ in $Q^{\HH}(x)$.

Given an $(r+1)$-uniform $\ell$-part hypergraph $\HH$, we define $\HH$ to be $(r,\eta,\gamma)$-\emph{nice} if 
\begin{enumerate}
\item[(i)] $\HH$ is $(r-1,\eta)$-near-regular,
\item[(ii)] the density of $\HH$ is between $\gamma/\ell^2$ and $3\gamma/\ell^2$,
\item[(iii)] $\Delta_r\le \gamma N$, and
\item[(iv)] $c_r^{\HH}\ge \gamma N^{r+1}/\ell^{r+1}$.
\end{enumerate}

Proposition~\ref{prop:LB} follows from the following two propositions.  The first states the required lower bound for this family of hypergraphs.

\begin{prop}\label{prop:LBif} If $\HH$ is an $(r+1)$-uniform $\ell$-part hypergraph (for $\ell=O_r(1)$) with maximum $r$-degree $\Delta_r$, which is $(r,\eta,\gamma)$-nice for $\eta=O_{r}(1/\gamma N)$ then
\[
\pr{D^{\HH}(B_m) > a}\, \geq\, N^{-O_{r,\tau}(1)} \exp \left( \frac{-O_{r,\tau}(1)a^{2/r}}{m \Delta_r^{2/r}} \right),
\]
for all $\max \{h/N, \eta h \}\le a\le \Omega_{r,\tau}(1)h$ and all $\tau N\le m\le (1-\tau)N$.
\end{prop}

The second proposition states that such hypergraphs exist.

\begin{prop}\label{prop:LBexists} For all $r\ge 2$ there exists $\ell=\ell(r)$ such that for all sequences $10/n\le \gamma_n\le \ell^{-2}$ for infinitely many values of $N$, there exists an $N$-vertex $(r+1)$-uniform hypergraph $\HH$ which is $\ell$-part and $(r,\eta,\gamma_N)$-nice with $\eta=O_r(1/\gamma_N N)$.
\end{prop}

The proof of Proposition~\ref{prop:LBexists} is technical and unlikely to be of general interest and so will be given in the appendix.

Our proof of Proposition~\ref{prop:LBif} is based on~\eqref{eq:LB} and the following two lemmas.  The first provides a bound on $\pr{E^{\ell}_{\eps}}$.

\begin{lemma}\label{lem:E} Let $\tau\in (0,1/2)$ and $\ell\in \mathbb{N}$.  Provided $\tau N \le m\le (1-\tau)N$ we have 
$$
\pr{E^{\ell}_{\eps}} \, \ge \, N^{-O_{\ell,\tau}(1)}\exp (-O_{\ell,\tau}(\eps^2 m))
$$
for all $0<\eps\le \tau/2(1-\tau)$.
\end{lemma}

The second lemma shows the effect of the uneven distribution (given by $E^{\ell}_{\eps}$) on $N^{\HH}(B_m)$ (or at least its expected value).  We show that for a certain value of $\eps$ the conditional expectation exceeds the unconditioned expectation by at least $2a$.

\begin{lemma}\label{lem:effect} There are constants $C_1=C_1(\ell,\tau)$ and $C_2 = C_2(\ell,\tau)$ such that if $\HH$ is an $(r+1)$-uniform $(r,\eta,\gamma)$-nice $\ell$-part hypergraph, then for all $ \max \{ h/N, \eta h\}\le a\le h$, we have
\[
\Ex{N^{\HH}(B_m)\, |\, E^{\ell}_{\eps}}\, \ge\, L^{\HH}(m)\, +\, 2a\, 
\]
for all $C_1 (a/h)^{1/r} \le \eps \le C_2 $.
\end{lemma}

We now deduce Proposition~\ref{prop:LBif} from these two lemmas.

\begin{proof}[Proof of Proposition~\ref{prop:LBif}]  We begin with an easy observation: if a random variable $X$ has mean $\mu$, has $|X|\le h$ and has $\Ex{X|E}\ge \mu+2a$ for some event $E$ then
\[
\pr{X> \mu+a}\, \ge \, \frac{a}{h} \pr{E}\, .
\]
This follows immediately using that $X\le h 1_{X> \mu+a}\, +\, \mu \, +\, a$ and taking the conditional expectation.

Taking $E=E^{\ell}_{\eps}$ where $\eps=\Theta_{\ell,\tau}((a/h)^{1/r})$ is given by Lemma~\ref{lem:effect} (observe that if $\Omega_{r,\tau}(1)$ is sufficiently small in Proposition~\ref{prop:LBif}, the interval for $\eps$ in Lemma~\ref{lem:effect} is nonempty), and using that $h\le aN$, it follows from Lemma~\ref{lem:E} that
\[
\pr{D^{\HH}(B_m) > a} \, \geq \, N^{-O_{\ell,\tau}(1)}\exp (-O_{\ell,\tau}(\eps^2 m))\, =\, N^{-O_{\ell,\tau}(1)}\exp \left(\frac{-O_{\ell,\tau}(1) a^{2/r}m}{h^{2/r}}\right)\,  .
\]
Since $\ell=O_r(1)$, the dependence of the constants on $\ell$ is just a dependence on $r$.  Also, conditions (ii) and (iii) in the definition of $(r,\eta,\gamma)$-nice give us that $h=\Omega_{\tau}(m^{r}\Delta_r)$ and so the required bound follows immediately.\end{proof}

We now turn to the proofs of Lemmas~\ref{lem:E} and~\ref{lem:effect}.  The proof of Lemma~\ref{lem:E} is a relatively straightforward exercise with Stirling's approximation, we include it for completeness.

\begin{proof}[Proof of Lemma~\ref{lem:E}]
Note first that the conditions imposed on $\eps$ ensure that the event $E_{\eps}^{\ell}$ is nonempty. Denoting $t = m/N$, observe that
$$
\pr{E^{\ell}_{\eps}}\, =\, \dfrac{ \dbinom{s}{(1+2\eps)st} \dbinom{s}{(1-\eps)st}^2 \dbinom{s}{st} ^{\ell - 3}}{\dbinom{s \ell}{ts\ell}}.
$$
Using Stirling's approximation for the factorials, if $k = \Omega(n)$, we have
$$
\dbinom{n}{k} \, = \, (1+o(1)) \sqrt{\frac{n}{2\pi k(n-k)}} \cdot \frac{n^n}{k^k(n-k)^{n-k}}.
$$
Note that if $k = cn$, where $c \in (0,1)$, we have
$$
\dbinom{n}{k} \, =\, (1+o(1))\Theta_{c}(n^{-1/2}) \cdot \exp(nH(c)),
$$
where $H(x) \, =\, -x \log x - (1-x) \log (1-x)$. Then
$$
\pr{E^{\ell}_{\eps}} \, =\, (1+o(1)) \Theta_{\ell, \tau}(s^{(1-\ell)/2}) \cdot \exp[s(H((1+2\eps) t) + 2H((1-\eps) t) -3H(t))].
$$
Using Taylor series, $s  = m/t\ell$ and $\tau \le t \le 1 - \tau$, we obtain that 
$$\pr{E^{\ell}_{\eps}}\, \ge\, N^{-O_{\ell,\tau}}\exp (-O_{\ell,\tau}(\eps ^2 m)),$$
as desired.
\end{proof}

We will need to work a little more to prove Lemma~\ref{lem:effect}.  In particular we shall use the following lemma about the coefficients of $Q^{\HH}(x)$ in the case $\HH$ is $(r,\eta,\gamma)$-nice.  Recall that $L^{\HH}(m) = \Ex{N^{\HH}(B_m)}$.

\begin{lemma}\label{lem:coefficients}
Let $\HH$ be an $(r+1)$-uniform $\ell$-part hypergraph which is $(r,\eta,\gamma)$-nice and let $m = tN$, where $\tau\le t\le 1-\tau$.  Then:
\begin{enumerate}[(i)]
\item $c_0^{\HH} \geq L^{\HH}(m)/t^{r+1}$;
\item $|c_j^{\HH}| \, =\, O_{r}(\eta h + h/N)$, if $1 \leq j \leq r-1$; 
\item $|c_{r+1}^{\HH}| \, =\, O_{r}(\Delta_r N^r)$.
\end{enumerate}
\end{lemma}

\begin{proof}
(i) For the first part, observe that $c_0^{\HH} = h$ and $L^{\HH}(m) = h \frac{(m)_{r+1}}{(N)_{r+1}}$. Because
$$
\frac{(m)_{r+1}}{(N)_{r+1}} \leq \left( \frac{m}{N} \right)^{r+1} \, =\, t^{r+1},
$$ 
it follows that $c_0^{\HH} \geq L^{\HH}(m)/t^{r+1}$. \\* \\*
(ii) In this part, we will use that $\HH$ is $(r-1, \eta)$-near-regular.  We first observe that
\begin{equation*}
\begin{split}
c_j^{\HH} \, &=\, \sum_{e \in E(\HH)} \sum_{i=0}^j (-1)^i 2^{j-i} \binom{e_2 + e_3}{i} \binom{e_1}{j-i} \\
&=\, \sum_{i=0}^j (-1)^i 2^{j-i} \sum_{e \in E(\HH)} \binom{e_2 + e_3}{i} \binom{e_1}{j-i}.
\end{split}
\end{equation*}
Note that by double counting,
$$
\sum_{e \in E(\HH)} \binom{e_2 + e_3}{i} \binom{e_1}{j-i}\, = \,
\sum_{\substack{A \subset V_1, |A| = j-i \\ B \subset V_2 \cup V_3, |B| = i}} d_j(A \cup B).
$$
Because $\HH$ is $(r-1, \eta)$-near-regular and $1 \leq j \leq r-1$, $d_j(A \cup B) \in (1 \pm \eta) \bar{d}_j^{\HH}$, which gives us
$$
\sum_{e \in E(\HH)} \binom{e_2 + e_3}{i} \binom{e_1}{j-i}\, \in\, \binom{s}{j-i} \binom{2s}{i} \bar{d}_j^{\HH} (1 \pm \eta).
$$
Using that $\bar{d}_j^{\HH} = O_{r}(h/N^j)$, we have
$$
\sum_{e \in E(\HH)} \binom{e_2 + e_3}{i} \binom{e_1}{j-i} \, =\, \binom{s}{j-i} \binom{2s}{i} \bar{d}_j^{\HH} \pm O_{r}(\eta h).
$$
Therefore
$$
c_j^{\HH} \, =\, \sum_{i=0}^j (-1)^i 2^{j-i} \binom{s}{j-i} \binom{2s}{i} \bar{d}_j^{\HH} \pm O_{r}(\eta h).
$$
The sum
$$
\sum_{i=0}^j (-1)^i 2^{j-i} \binom{s}{j-i} \binom{2s}{i}
$$
is a polynomial in $s$ of degree at most $j$. However the coefficient of $s^j$ in this polynomial is
$$
\sum_{i=0}^j \frac{(-1)^i \cdot 2^{j-i} \cdot 2^i}{(j-i)!i!} \, =\, \frac{2^j}{j!} \sum_{i=0}^j (-1)^i \binom{j}{i} \, =\, 0.
$$
Hence the sum is actually a polynomial in $s$ of degree at most $j-1$. Using again that $\bar{d}_j^{\HH} = O_{r}(h/N^j)$ and using that $s = O_{r}(N)$, we obtain that $|c_j^{\HH}| = O_{r}(\eta h + h/N)$, as desired.
\\*
\\*
(iii) In this last part, as the hypergraph is $(r+1)$-uniform, $|c_{r+1}^{\HH}| = O_{r}(h)$ and since $h = O_r(\Delta_r N^r)$, it follows that $|c_{r+1}^{\HH}| = O_{r}(\Delta_r N^r)$.
\end{proof}

Armed with Lemma~\ref{lem:coefficients} we now prove Lemma~\ref{lem:effect} about the conditional expectation $\Ex{N^{\HH}(B_m)\, |\, E^{\ell}_{\eps}}$.

\begin{proof}[Proof of Lemma~\ref{lem:effect}] Let $\HH$ be a $(r,\eta,\gamma)$-nice $\ell$-part hypergraph and let $\eps=C_1(a/h)^{1/r}$ for a constant $C_1$ which we can choose later, as a function of $\ell$ and $\tau$.  We must prove that
\[
\Ex{N^{\HH}(B_m)\, |\, E^{\ell}_{\eps}}\, \ge\, L^{\HH}(m)\, +\, 2a\, .
\]
We first observe that
\begin{equation*}
\begin{split}
\Ex{N^{\HH}(B_m)|E^{\ell}_{\eps}} \,&=\, \sum_{e \in E(\HH)} \prod_{i=1}^{\ell} \frac{\binom{s-e_i}{|B_m \cap V_i|-e_i}}{\binom{s}{| B_m \cap V_i|}} \\
&= \, \sum_{e \in E(\HH)} \prod_{i=1}^{\ell} \frac{\binom{|B_m \cap V_i|}{e_i}}{\binom{s}{e_i}}.
\end{split}
\end{equation*}
As we are conditioned on the event $E^{\ell}_{\eps}$, we have $|B_m \cap V_1| = (1+2\eps)st\pm 1$, $|B_m \cap V_2| = |B_m \cap V_3| = (1- \eps)st\pm 1$ and $|B_m \cap V_i| = st\pm 1$ for all $4 \leq i \leq \ell$. Using that $(N)_k = N^k (1 \pm O(1/N))$, we obtain
$$
\Ex{N^{\HH}(B_m)|E^{\ell}_{\eps}} \, =\, t^{r+1}Q^{\HH}(\eps)\,  \pm \, O_r(h/N).
$$
By Lemma \ref{lem:coefficients}, choosing $C_2$ sufficiently small, there are positive constants $\alpha$ and $\beta$ depending on $r$ and $\tau$, such that
$$
\Ex{N^{\HH}(B_m)|E^{\ell}_{\eps}} \, \geq \, L^{\HH}(m) + \alpha \eps^r \Delta_r N^r - \beta (\eta h +h/N).
$$
As $\eps=C_1(a/h)^{1/r}$ and $\Delta_r \ge h/N^r$,
$$
\Ex{N^{\HH}(B_m)|E^{\ell}_{\eps}} \, \geq \, L^{\HH}(m) + \alpha C_1 a - \beta (\eta h +h/N).
$$
Since $a \ge \max \{ h/N, \eta h \}$, we can choose $C_1 = (2+2\beta)/\alpha $, obtaining
$$
\Ex{N^{\HH}(B_m)|E^{\ell}_{\eps}} \, \geq \, L^{\HH}(m) + 2a,
$$
as desired.
\end{proof}

\section{Concluding remarks and open questions}\label{sec:final}

We showed in Section~\ref{sec:LB} that Theorem~\ref{thm:nearreg} is best possible up to the implicit constants, at least when $m=\Theta(N)$.  However, for particular cases, such as the $3$-term arithmetic progressions in $\Z_N$ we do not believe that the bounds obtained are best possible.  A lower bound of the form
\[
\pr{D^{3}(B_m)>a}\, \ge\, \exp\left(\frac{-O(1)a^{2/3}N^{2/3}}{m}\right)
\]
may be proved by considering the probability that the interval $\{1,\dots ,\lfloor N/3\rfloor\}$ contains significantly more points than expected.  Specifically $(1+\eps)N/3$ points where $\eps\approx a^{1/3}/N^{2/3}$.  Writing $D^{3}(B_m)$ for the deviation of number of three term arithmetic progressions, it remains an interesting open problem to determine the value of
\[
\log{\pr{D^{3}(B_m)>a}}
\]
up to a constant factor.  We stress that this problem is open even in the dense case $m=\Theta(N)$, for $N\ll a\ll N^2$.

Finally, we stress that we do not believe that our result is best possible in the sparse case, $m=o(N)$.  For example, in the case of regular $3$-uniform hypergraphs we believe it is possible to improve the bound
\[
\pr{D^{\HH}(B_m)>a}\, \le\, N^{O_k(1)}\, \exp\left(\frac{-\Omega_k(1) a}{m\Delta_{2}}\right)
\]
to
\[
\pr{D^{\HH}(B_m)>a}\, \le\, N^{O_k(1)}\, \exp\left(\frac{-\Omega_k(1) aN^{1/2}}{m^{3/2}\Delta_{2}}\right)
\]
by using Freedman's inequality in place of the Azuma--Hoeffding inequality at a certain point.  It would be of interest to determine the best possible result across the whole range of sparse densities.  In particular we do not know whether the stronger bound
\[
\pr{D^{\HH}(B_m)>a}\, \le\, N^{O_k(1)}\, \exp\left(\frac{-\Omega_k(1) aN}{m^{2}\Delta_{2}}\right)
\]
may hold in general for regular $3$-uniform hypergraphs.  This final bound, if true, would be best possible up to the implicit constants.

For simplicity we have stated our open problems for $3$-uniform hypergraphs.  The analogous problems for $k$-uniform hypergraphs with one of our regularity properties are also open.

Improvements in results for the $m$-model $B_m$ for $m=o(N)$ would almost certainly allow one to extend the range of deviations covered by Theorem~\ref{thm:pworld} for the $p$-model $B_p$.



\begin{thebibliography}{99}

\bibitem{A1967} K. Azuma, { \em Weighted sums of certain dependent random variables}, Tohoku Math. J. \textbf{19} (1967), no. 3, 357--367.

\bibitem{B1960} R.R. Bahadur, {\em Some approximations to the binomial distribution function}, Ann. Math. Statist. \textbf{31} (1960), no. 1, 43--54.

\bibitem{BMS2014} J. Balogh, R. Morris and W. Samotij, {\em Random sum-free subsets of abelian groups}, Israel J. Math. \textbf{199} (2014), no. 2, 651--685. 

\bibitem{BM2019} B.B. Bhattacharya and S. Mukherjee, { \em A Note on Replica Symmetry in Upper Tails of Mean-Field Hypergraphs}, {\tt arXiv:1812.09841}.

\bibitem{BGSZ2016} B.B. Bhattacharya, S. Ganguly, X. Shao, and Y. Zhao, { \em Upper tails for arithmetic progressions in a random set},  Int. Math. Res. Not. \textbf{1} (2020), 167--213. 

\bibitem{BSS2020} R. Berkowitz, A. Sah and M. Sawhney,  {\em Number of arithmetic progressions in dense random subsets of $\ZZ/n\ZZ$}, {\tt arXiv:1907.11807}

\bibitem{B2013} S. Boucheron, G. Lugosi and P. Massart, { \em Concentration inequalities}, Oxford University Press, Oxford, 2013.

\bibitem{BG2020} J. Bri\"et and S. Gopi, {\em Gaussian width bounds with applications to arithmetic progressions in random settings},  Int. Math. Res. Not. \textbf{22} (2020), 8673--8696. 

\bibitem{CV2011} S. Chatterjee and S.R.S. Varadhan, {\em The large deviation principle for the Erd\"os-R\'enyi random graph},  European J. Combin. \textbf{32} (2011), no. 7, 1000--1017. 

\bibitem{C2016} S. Chatterjee, { \em An introduction to large deviations for random graphs}, Bull. Amer. Math. Soc. (N. S.) \textbf{53} (2016), no. 4, 617--642.

\bibitem{DKLRS2018} D. Dellamonica, Y.  Kohayakawa, S.J. Lee, V. R\"odl and W. Samotij, {\em  The number of $B_h$--sets of a given cardinality},  Proc. Lond. Math. Soc. (3) 116 (2018), no. 3, 629--669.

\bibitem{DE2009} H.~D\"oring and P.~Eichelsbacher, {\em Moderate deviations in a random graph and for the spectrum of Bernoulli random matrices}, Electron. J. Probab. \textbf{14} (2009), 2636--2656.

\bibitem{FMN2016} V.~F\'eray, P.L.~M\'eliot and A.~Nikeghbali, {\em Mod-$\phi$ convergence I: Normality zones and precise deviations}, Springer Briefs in Probability and Mathematical Statistics, Springer, 2016.


\bibitem{GGS2019} C. Goldschmidt, S. Griffiths and A. Scott, {\em Moderate deviations of subgraph counts in the Erd\H os-R\'enyi random graphs $G(n,m)$ and $G(n,p)$}, Trans. Amer. Math. Soc. \textbf{373} (2020), 5517--5585.

\bibitem{GKS2020} S. Griffiths, C. Koch and M. Secco,  {\em Deviation probabilities for arithmetic progressions and irregular discrete structures}, {\tt arXiv:2012.09280}

\bibitem{GRR1996} R. Graham, V. R\"{o}dl and A. Ruci\'{n}ski, { \em On Schur properties of random subsets of integers}, J. Number Theory \textbf{61} (1996), 388--408.

\bibitem{HMS2019} M. Harel, F. Mousset and W. Samotij, {\em Upper tails via high moments and entropic stability}, {\tt arXiv:1904.08212}.

\bibitem{H1963} W. Hoeffding, { \em Probability inequalities for sums of bounded random variables}, J. Amer. Statist. Assoc. \textbf{58} (1963), 13--30.


\bibitem{JOR2002}S. Janson, K.Oleszkiewicz and A. Ruci\'nski, {\em Upper tails for subgraph counts in random graphs}, Israel J. Math. \textbf{142} (2004), 61--92.

\bibitem{JR1988} S. Janson and A. Ruci\'nski, {\em When are small subgraphs of a random graph normally distributed?}, Probab. Theory Related Fields \textbf{78} (1988), no. 1, 1--10.

\bibitem{JR2011} S. Janson and A. Ruci\'nski, { \em Upper tails for counting objects in randomly induced subhypergraphs and rooted random graphs}, Ark. Mat. \textbf{49} (2011), 79--96.

\bibitem{JW2015} S.~Janson and L.~Warnke, {\em The lower tail: Poisson approximation revisited}, Random Structures Algorithms \textbf{48} (2015), no.2, 219--246.


\bibitem{KV2000} J.H.Kim and V.H. Vu, {\em Concentration of multivariate polynomials and its applications}, Combinatorica \textbf{20} (2000), no. 3, 417--434.

\bibitem{K2013} Y. Kohayakawa,  S. J. Lee,   V. R\"odl and   Wojciech Samotij, {\em The number of Sidon sets and the maximum size of Sidon sets contained in a sparse random set of integers}, Random Structures Algorithms \textbf{46} (2015), no. 1, 1--25. 

\bibitem{RR1997} V. R\"odl and A. Ruci\'nski, {\em Rado partition theorem for random subsets of integers} Proc. London Math. Soc. \textbf{74}, (1997), 481--502.

\bibitem{RSV2015} J. Ru\'{e}, O. Serra and L. Vena, { \em Counting configuration-free sets in groups}, European J. Combin. \textbf{66} (2017), 281--307. 

\bibitem{RSZ2018} J. Ru\'{e}, C. Spiegel and A. Zumalac\'arregui, {\em Threshold functions and Poisson convergence for systems of equations in random sets}  Math. Z. \textbf{288} (2018), no. 1-2, 333--360.

\bibitem{R1993} I.Z. Ruzsa, {\em Solving a linear equation in a set of integers I}, Acta Arithmetica LXV.3 (1993), 259--282.


\bibitem{S2016} M. Schacht, { \em Extremal results for random discrete structures}, Ann. of Math. (2) \textbf{184} (2016), 333--365.


\bibitem{W2016}  L. Warnke, { \em On the method of typical bounded differences}, Combin. Probab. Comput. \textbf{25}
(2016), no. 2, 269--299.

\bibitem{W2017} L. Warnke, {\em Upper tails for arithmetic progressions in random subsets}, Israel J. Math. \textbf{221} (2017) 221-- 317. 

\bibitem{W2020} L. Warnke, {\em On the missing log in upper tail estimates}, J. Combin. Theory Ser. B \textbf{140} (2020), 98--146.



\end{thebibliography}

\section{Appendix: Lower bound construction}
 
In this appendix we construct the hypergraphs whose existence is claimed by Proposition~\ref{prop:LBexists}.  We shall recall the necessary definitions and state a proposition (Proposition~\ref{prop:appendix}) which implies Proposition~\ref{prop:LBexists}.

Let us recall that we call a hypergraph $\HH$ an $\ell$-part hypergraph if its vertex set $V=[N]$ may be partitioned $V=V_1\cup \dots \cup V_{\ell}$ into $\ell$ equal parts of size $s:=N/\ell$ in such a way that when we label the vertices $v_{i,j}:i\in [\ell],j\in [s]$ the hypergraph $\HH$ is invariant under permutation of the parts $i\in [\ell]$.

Recall also that we define the polynomial
$$
Q^{\HH}(x) \, =\, \sum_{e \in E(\HH)} (1+2x)^{e_1} (1-x)^{e_2 + e_3}
$$
where $e_i:=|e\cap V_i|$ for $i\in [\ell]$, and that $c_j^{\HH}$ denotes the coefficient of $x^j$ in $Q^{\HH}(x)$.

Finally, we defined an $(r+1)$-uniform $\ell$-part hypergraph $\HH$ to be $(r,\eta,\gamma)$-\emph{nice} if 
\begin{enumerate}
\item[(i)] $\HH$ is $(r-1,\eta)$-near-regular,
\item[(ii)] the density of $\HH$ is between $\gamma/\ell^2$ and $3\gamma/\ell^2$,
\item[(iii)] $\Delta_r\le \gamma N$, and
\item[(iv)] $c_r^{\HH}\ge \gamma N^{r+1}/\ell^{r+1}$.
\end{enumerate}

The proposition we must prove states that such hypergraphs exist.  Specifically, for all $r\ge 2$ there exists $\ell=\ell(r)$ such that for all sequences $\Theta(1/n)\le \gamma_n\le \ell^{-2}$ there are infinitely many $(r+1)$-uniform $\ell$-part hypergraphs $\HH$ which are $(r,\eta,\gamma_N)$-nice with $\eta=O_r(1/\gamma_N N)$.

To this end, let us fix $r\ge 2$.  We shall work with $\ell$ throughout the definition without fixing its value.  It will be clear at the end of the proof that if $\ell$ is taken sufficiently large then for all sufficiently large multiples $N=s\ell$ of $\ell$ the construction satisfies all 4 conditions above.  With this in mind we may always assume that $s=N/\ell$ is sufficiently large.

We work with a fixed value of $\gamma\in [10\ell^2/s,1/2]$ and the example we provide will contain $\approx 2\gamma N^{r+1}/(r+1)!\ell^{2}$ edges.  This gives us the claimed range of densities.

We now state more precisely the type of construction we shall give.  In particular we shall fix the choice of $\ell$ as $\ell=4(r+1)!$.

\begin{prop}\label{prop:appendix}
Let $r\ge 2$, and fix $\ell=4(r+1)!$.  There exists a constant $C=C(r)$ such that for all $0\le \gamma\le 1/2$ and all $s\ge 10\ell^2/\gamma$ there exists an $\ell$-part $(r+1)$-uniform hypergraph $\HH$ on $s\ell$ vertices which is $(r,C/\gamma s,\gamma)$-nice.
\end{prop}

It is easily checked that Proposition~\ref{prop:LBexists} follows from this proposition.

A major challenge in the proof of Proposition~\ref{prop:appendix} is to ensure that the hypergraph $\HH$ we construct is $(r-1,C/\gamma s)$-near-regular.  This is more difficult than one might imagine.  We shall define $\HH$ with enough symmetry that any two sets with the same intersection pattern with the parts have (essentially) the same degree in $\HH$.  This reduces the number of degrees we need to check to a finite number (the number of partitions of the number $r-1$).  In order to formalise these details, and define $\HH$, we introduce the concept of \emph{type}.

\subsection{Types and one-type hypergraphs}

We continue to use the partition $V=V_1 \cup \dots V_{\ell}$ into $\ell$ parts defined above.  Given a set $e$ of elements of $V$ consider the vector $(e_1,\dots ,e_\ell):=(|e\cap V_1|,|e\cap V_2|,\dots ,|e\cap V_\ell|)$.  We define $\x(e)$, the \emph{type} of $e$, to be the vector obtained from $(e_1,\dots ,e_\ell)$ by placing its entries in decreasing order and removing the $0$s.  For example

\begin{center}
\definecolor{qqqqff}{rgb}{0.,0.,1.}
\begin{tikzpicture}[scale = 0.75,line cap=round,line join=round,>=triangle 45,x=1.0cm,y=1.0cm]
\clip(0.5707978963185597,-3.9187528174304935) rectangle (9.135786626596554,2.2570548459804654);
\draw [rotate around={-166.7572837874792:(2.087858889133701,0.9116157159953564)},line width=1.pt] (2.087858889133701,0.9116157159953564) ellipse (1.4126982135558093cm and 1.2084591543887162cm);
\draw [rotate around={-164.69338833613045:(2.310351311732481,-2.5178899656394145)},line width=1.pt] (2.310351311732481,-2.5178899656394145) ellipse (1.4085878236016962cm and 1.2202846649203718cm);
\draw [rotate around={-164.69338833613136:(7.566549658839922,0.8540108608068558)},line width=1.pt] (7.566549658839922,0.8540108608068558) ellipse (1.4085878236017062cm and 1.2202846649203827cm);
\draw [rotate around={-164.69338833613125:(7.450847179501085,-2.3526007094410653)},line width=1.pt] (7.450847179501085,-2.3526007094410653) ellipse (1.408587823601709cm and 1.2202846649203716cm);
\draw [rotate around={6.555187897679517:(4.0887084479498625,-0.3916036783241044)},line width=1.pt,color=qqqqff] (4.0887084479498625,-0.3916036783241044) ellipse (3.120505232255373cm and 1.5114241791975713cm);
\begin{scriptsize}
\draw [fill=black] (1.495966942148762,1.4209917355371895) circle (1.5pt);
\draw [fill=black] (2.619933884297523,1.602809917355371) circle (1.5pt);
\draw [fill=black] (1.9231645379414006,0.16839969947408268) circle (1.5pt);
\draw [fill=black] (2.851338842975209,0.4623140495867776) circle (1.5pt);
\draw [fill=black] (6.9569737039819755,1.2953719008264475) circle (1.5pt);
\draw [fill=black] (7.903630353117965,1.4005559729526682) circle (1.5pt);
\draw [fill=black] (6.896868519909849,0.24353117956424034) circle (1.5pt);
\draw [fill=black] (7.903630353117965,0.48395191585274483) circle (1.5pt);
\draw [fill=black] (1.772901577761085,-2.085544703230647) circle (1.5pt);
\draw [fill=black] (2.7946897069872314,-1.694861006761827) circle (1.5pt);
\draw [fill=black] (1.878085649887306,-3.0322013523666334) circle (1.5pt);
\draw [fill=black] (3.0351104432757365,-2.987122464312539) circle (1.5pt);
\draw [fill=black] (6.821737039819688,-1.9202554470323001) circle (1.5pt);
\draw [fill=black] (6.950512396694218,-3.0748760330578446) circle (1.5pt);
\draw [fill=black] (8.272826446280995,-1.7525619834710695) circle (1.5pt);
\draw [fill=black] (8.305884297520665,-2.9261157024793323) circle (1.5pt);
\end{scriptsize}
\end{tikzpicture}
\qquad
\begin{tikzpicture}[scale = 0.75,line cap=round,line join=round,>=triangle 45,x=1.0cm,y=1.0cm]
\clip(0.3971188447194279,-3.86889056604115) rectangle (9.505417170776743,2.2374766396917947);
\draw [rotate around={-166.7572837874792:(2.087858889133701,0.9116157159953564)},line width=1.pt] (2.087858889133701,0.9116157159953564) ellipse (1.4126982135558093cm and 1.2084591543887162cm);
\draw [rotate around={-164.69338833613045:(2.310351311732481,-2.5178899656394145)},line width=1.pt] (2.310351311732481,-2.5178899656394145) ellipse (1.4085878236016962cm and 1.2202846649203718cm);
\draw [rotate around={-164.69338833613136:(7.566549658839922,0.8540108608068558)},line width=1.pt] (7.566549658839922,0.8540108608068558) ellipse (1.4085878236017062cm and 1.2202846649203827cm);
\draw [rotate around={-164.69338833613125:(7.450847179501085,-2.3526007094410653)},line width=1.pt] (7.450847179501085,-2.3526007094410653) ellipse (1.408587823601709cm and 1.2202846649203716cm);
\draw [rotate around={-2.0876178374583:(4.847851689871737,-0.7165059185471561)},line width=1.pt,color=qqqqff] (4.847851689871737,-0.7165059185471561) ellipse (3.0362501528607324cm and 1.6743801568829817cm);
\begin{scriptsize}
\draw [fill=black] (1.495966942148762,1.4209917355371895) circle (1.5pt);
\draw [fill=black] (2.619933884297523,1.602809917355371) circle (1.5pt);
\draw [fill=black] (1.9231645379414006,0.16839969947408268) circle (1.5pt);
\draw [fill=black] (2.851338842975209,0.4223140495867776) circle (1.5pt);
\draw [fill=black] (6.9569737039819755,1.2953719008264475) circle (1.5pt);
\draw [fill=black] (7.903630353117965,1.4005559729526682) circle (1.5pt);
\draw [fill=black] (6.896868519909849,0.24353117956424034) circle (1.5pt);
\draw [fill=black] (7.903630353117965,0.48395191585274483) circle (1.5pt);
\draw [fill=black] (1.772901577761085,-2.085544703230647) circle (1.5pt);
\draw [fill=black] (2.7946897069872314,-1.694861006761827) circle (1.5pt);
\draw [fill=black] (1.878085649887306,-3.0322013523666334) circle (1.5pt);
\draw [fill=black] (3.0351104432757365,-2.987122464312539) circle (1.5pt);
\draw [fill=black] (6.721737039819688,-1.8202554470323001) circle (1.5pt);
\draw [fill=black] (6.950512396694218,-3.0748760330578446) circle (1.5pt);
\draw [fill=black] (8.272826446280995,-1.7525619834710695) circle (1.5pt);
\draw [fill=black] (8.305884297520665,-2.9261157024793323) circle (1.5pt);
\end{scriptsize}
\end{tikzpicture}

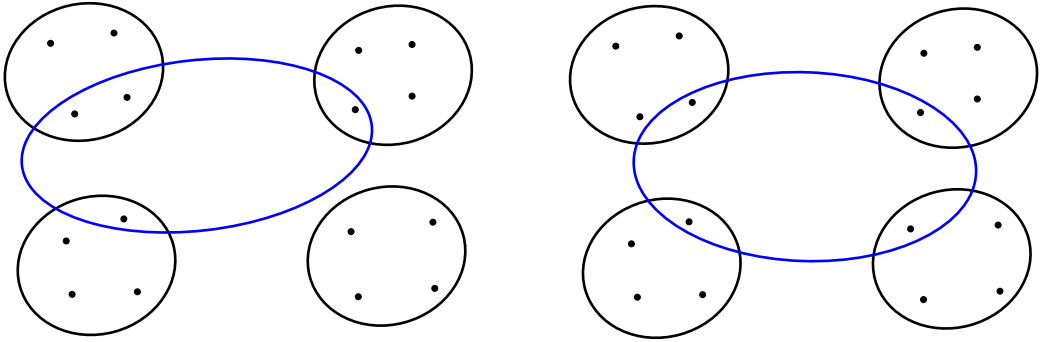
\captionof{figure}{An edge of type $(2,1,1)$ and an edge of type $(1,1,1,1)$}
\end{center}
Note that the set of possible types of edges in our $(r+1)$-uniform hypergraph $\HH$ is $\PP_{r+1}$, the set of partitions of $r+1$.  We now define certain families of one-type hypergraphs.  Given a type $\x\in \PP_{r+1}$ we may define $\HH_{\x}$ to be the hypergraph consisting of all edges of type $\x$.

We shall also consider sparser subhypergraphs of $\HH_\x$.  We remark that each vertex $v\in V$ may be viewed as a pair $(i_v,j_v)$ where $i_v\in [\ell]$ and $j_v\in [s]$.  Given $\alpha\in (0,1)$ we say that a sequence $(i_1,j_1),\dots ,(i_{r+1},j_{r+1})$ of elements of $V$ is $\alpha$-good if 
\[
j_1\, +\dots +\, j_{r+1}\, \in \{1,\dots ,\lfloor\alpha s\rfloor\}\qquad (\text{mod } s)\, .
\]
We may define $\HH^{\alpha}_{\x}$ to be the hypergraph obtained from $\HH_\x$ by keeping those edges $e=\{(i_1,j_1),\dots ,(i_{r+1},j_{r+1})\}$ which are $\alpha$-good.  It will also be useful to define $\bar{\HH}^{\alpha}_{\x}$ to be the equivalent hypergraph, but in which multi-sets are also permitted.  That is $\bar{\HH}^{\alpha}_{\x}$ contains those multi-sets of $r+1$ vertices which have type $\x$ and are $\alpha$-good.

The idea of this sparsification is that roughly $\alpha$ proportion of the edges of $\HH_{\x}$ remain.  In fact the same is true for degrees and higher degrees, up to the $r$-degree which will satisfy $\Delta_r\, =\, O(\alpha s)$, where the constant in the $O(\cdot)$ may depend on $r$ and $\ell$.

\subsection{Multi-type hypergraphs and degrees:} We construct our hypergraph $\HH$ as a union of one-type hypergraphs $\HH^{\alpha}_{\x}$ defined above.  In fact we do not need to consider all types.  We consider only the type $(r+1)$ and types of the form $(\x,1,1)$ where $\x\in \PP_{r-1}$.  For each vector $\x\in \PP_{r-1}$ we will define a vector $\x^+\in \PP_{r+1}$ that ``extends'' $\x$.  For $\x=(r-1)$ we define $\x^+=(r+1)$, and all other $\x\in \PP_{r-1}$ we define $\x^+=(\x,1,1)$.   Given $\valpha=(\alpha_\x:\x\in \PP_{r-1})$, which assigns a value $\alpha_{\x}\in (0,1)$ to each partition $\x$ of $r-1$, we define the hypergraph
\[
\HH^{\valpha}\, :=\, \bigcup_{\x\in \PP_{r-1}}\HH^{\alpha_x}_{\x^+}\, .
\]
To re-iterate, the hypergraph $\HH^{\valpha}$ is a union of hypergraphs $\HH^{\alpha}_{\y}$ where the type $\y$ is either $(r+1)$ or of the form $(\x,1,1)$ for some $\x\in \PP_{r-1}$.  We may think of the $\alpha_{\x}$ as ``weights'' which are associated with the types $\x^{+}$.  We may also define $\bar{\HH}^{\valpha}$ to be the equivalent in which multi-sets are permitted.

We now study degrees of vertices and sets in $\HH^{\valpha}$.  These degrees depend on the weights $\alpha_{\x}$ in a predictable manner, and we shall choose the weights to ensure $\HH^{\valpha}$ is $(r-1,C/\gamma s)$-near-regular.  It will be extremely helpful to us that the degree of an $(r-1)$ set $A$ of type $\x\in \PP_{r-1}$ depends strongly on $\alpha_{\x}$ and weakly on all of the other weights $\alpha_{\y}$.  Let us first study the degree of the type $\x$ abstractly, and then relate this to the degree of $A$ itself.  

Given $\x\in\PP_{r-1}$ let $|x|$ denote the number of entries in the sequence $\x$.  A prototypical set of type $\x$ is the set $A^0_{\x}$ which contains the elements $(1,1),\dots ,(1,x_1)$ of $V_1$ the elements $(2,1),\dots ,(2,x_2)$ of $V_2$ and so on up to $(a,1),\dots (a,x_a)$, where $a:=|x|$.  The abstract degree of $\x$ will be a slight modification of the degree of $A^0_{\x}$ in $\HH^{\valpha}$ in that we will work with the multi-set version $\bar{\HH}^{\valpha}$.  We define the abstract degree $d^{\valpha}_{\x}$ of type $\x$ to be the number of pairs $u,u'\in V$ such that $A^0_{\x}\cup \{u,u'\}\in \bar{\HH}^{\valpha}$.

With the following two lemmas we prove that the degrees of sets $A$ of type $\x$ are very well approximated by $d^{\valpha}_{\x}$ and that $d^{\valpha}_{\x}$ itself is very well controlled in terms of the weight $\alpha_{\x}$.

\begin{lemma}\label{lem:setdegs} Let $\x\in \PP_{r-1}\setminus \{(r-1)\}$ and let $A\subseteq V$ be a subset of $r-1$ vertices of type $\x$.  Let $\valpha=(\alpha_{\x}:\x\in \PP_{r-1})$ and let $\alpha'$ be the largest value of $\alpha_{\y}$ for $\y\in \PP_{r-1}\setminus \{(r-1)\}$.  Then the degree of $A$ in $\HH^{\valpha}$ differs from $d^{\valpha}_{\x}$ by at most $O(\alpha' n)$.
\end{lemma}

\begin{proof} By the symmetry of the definition of $\bar{\HH}^{\valpha}$, the degree of $A$ in $\bar{\HH}^{\valpha}$ is precisely $d^{\alpha}_{\x}$.  The degree of $A$ in $\HH^{\valpha}$ may be slightly smaller as we no longer count multi-sets.  It therefore suffices to bound the number of $\alpha'$-good multi-sets of $r+1$ vertices which contain $A$.  There are $r$ choices of which element to repeat (one of the $r-1$ elements of $A$ or the new $r$th element) and in each case only $\alpha_{\x} n\le \alpha'n$ choices of the new $r$th element such that the resulting multi-set is $\alpha_{\x}$-good.
\end{proof}

\begin{lemma}\label{lem:abstractdegs} Let $\x\in \PP_{r-1}$, and let $|x|=a\ge 2$.  Let $\valpha=(\alpha_{\x}:\x\in \PP_{r-1})$ and let $\alpha(a)$ be the largest value of $\alpha_{\y}$ for $\y\in \PP_{r-1}$ with $2\le |\y|\le a-1$.  Then the abstract degree $d^{\valpha}_{\x}$ of $\x$ satisfies
\[
\binom{\ell-a}{2} \alpha_{\x} s^2\, -\, O(s)\le\, d^{\valpha}_{\x} \, \le\, \binom{\ell-a}{2} \alpha_{\x} s^2\, +\, r^4\ell \alpha(a) s^2\, .
\]
\end{lemma}

\begin{proof} The abstract degree $d^{\valpha}_{\x}$ of $\x$ is the degree of $A^0_{\x}$ in $\bar{\HH}^{\valpha}$.  Since $\bar{\HH}^{\valpha}$ contains $\HH^{\alpha_\x}_{\x}$ we may prove the lower bound simply by observing that $A^0_{\x}$ is contained in $\binom{\ell-a}{2} \alpha_{\x} s^2 - O(s)$ edges of $\HH^{\alpha_\x}_{\x}$.  To see this note that the edges of $\HH^{\alpha_\x}_{\x}$ are of type $\x^{+}=(\x,1,1)$ so we may extend from $A^0_{\x}$ to an edge of type $\x^+$ by adding one vertex to each of two new parts of $V_1\cup\dots \cup V_{\ell}$, there are $\binom{\ell - a}{2}$ choices of the two other parts, $s$ choices of an element in one, and $\lfloor\alpha_{\x}s\rfloor =\alpha_{\x}s\pm 1$ choices of an element in the second such that the resulting set is $\alpha_{\x}$-good.

To prove the upper bound we must consider all the other $\bar{\HH}^{\alpha_\y}_{\y}$ subhypergraphs of $\bar{\HH}^{\valpha}$.  In fact, if $\y$ is not closely related to $\x$ then $A^0_{\x}$ will not be included in any edges of type $\y^{+}$ (for example an edge of type $(3,1)$ is not contained in any of type $(2,2,1,1)$).  In fact it is necessary that $x_i\le y_i$ for all $i=1,\dots ,a$, and $x_i\le y^{+}_i$ for all $i=1,\dots ,a$.  This last condition implies that $|\y|\ge a-2$.  The cases that $\y\in \PP_{r-1}$ and $x_i\le y^{+}_i$ for all $i=1,\dots ,a$ correspond to just three possible cases:
\begin{enumerate}
\item[(i)] $x_a=1$ and there is a unique $j<a$ such that $y_j=x_j+1$, and $y_i=x_i$ for $i\in [a-1]\setminus\{j\}$.
\item[(ii)] $x_{a-1}=x_a=1$ and there is a unique $j<a$ such that $y_j=x_j+2$, and $y_i=x_i$ for $i\in [a-2]\setminus\{j\}$.
\item[(iii)] $x_{a-1}=x_a=1$ and there is a pair $j<j'<a$ such that $y_j=x_j+1$, $y_{j'}=x_{j'}+1$ and $y_i=x_i$ for $i\in [a-2]\setminus\{j,j'\}$.
\end{enumerate}
It is easily verified that the number of ways to extend $A^0_{\x}$ to an $\alpha_{\y}$-good edge of type $\y$ is at most $r\ell\alpha' s^2$ in case (i), at most $r\alpha' s^2$ in case (ii) and at most $r^2\alpha' s^2$ in case (iii).  Since each case corresponds to a bounded number of sequences $\y$ (certainly at most $r^2$ in each case) we deduce the upper bound.
\end{proof}

In fact, there are some particular types $\x$ which we shall analyse particularly carefully so it is useful to count the degree even more carefully in these cases.

\begin{lemma}\label{lem:12} If $\x=(r-1)$ then every $(r-1)$-tuple of type $\x$ has degree
\[
\frac{1}{2}\lfloor \alpha_{\x}s^2 \rfloor\, +\, O(s)\, .
\]
If $|\x|=2$ then every $(r-1)$-tuple of type $\x$ has degree
\[
\binom{\ell-2}{2}\alpha_{\x}s^2\, +\, O(s)\, .
\]
\end{lemma}

\begin{proof} First, for $\x=(r-1)$, an $(r-1)$-tuple is of type $\x$ if it consists of $r-1$ vertices in the same part.  Without loss of generality consider $(1,j_1),\dots ,(1,j_{r-1})$.  The only edges of $\HH^{\valpha}$ which contain such a tuple are the edges of type $\x^+=(r+1)$, which consist of $r+1$ elements in the same part.  There are clearly $\binom{s-r+1}{2}$ ways to extend to an $(r+1)$ set of this type.  The number of choices which produce an $\alpha_{\x}$-good edge is between
\[
\frac{1}{2}s\lfloor \alpha_{\x}s \rfloor\, -\, r\alpha_{\x}s \quad \text{and}\quad \frac{1}{2}s \alpha_{\x}s\, .
\]
Since both these expressions are of the form
\[
\frac{1}{2}s\lfloor \alpha_{\x}s \rfloor\, +\, O(s)
\]
this completes the proof of the first part.

Now consider $\x$ with $|\x|=2$.  This means that $\x=(r-1-i,i)$ for some $1\le i\le \lfloor (r-1)/2\rfloor$.  It may be easily checked in these cases that an $(r-1)$-tuple of type $\x$ is not contained in any $\y^{+}$ for $\y\neq \x$, and so the degree of this tuple in $\HH^{\valpha}$ is the degree of this tuple in $\HH^{\alpha_{\x}}_{\x}$.  To complete the proof it suffices to observe that this degree is $\binom{\ell-2}{2}\alpha_{\x}s^2+ O(s)$.  This follows easily from the fact that there are $\binom{\ell - 2}{2}$ choices of the pair of parts of the last two vertices, $s$ choices for a vertex in the first part, and finally $\lfloor \alpha_{\x}s \rfloor=\alpha_{\x}s\pm 1$ choices of the second vertex, in such a way that an $\alpha_{\x}$-good edge is created.
\end{proof}

\subsection{Choice of the weights -- completing the proof of Proposition~\ref{prop:appendix}}

As we have discussed above we shall define the hypergraph $\HH$ as $\HH^{\valpha}$ for a particular choice of the vector of weights $\valpha$.  Our aim is to select the weights such that all $(r-1)$-tuples have degree approximately $\gamma s^2$.   To be more precise, all such degrees will be $\gamma s^2 \pm C s$.  This will be sufficient to ensure that the hypergraph $\HH$ satisfies conditions (i),(ii) and (iii) of the $(r,C/s,\gamma)$-nice definition.  Once this has been achieved we shall verify that condition (iv) also holds for $\HH$.

\begin{lemma}\label{lem:weight} If the constant $C$ is chosen sufficiently large then there is a choice of the weights $\valpha=(\alpha_{\x}:\x\in \PP_{r-1})$ such that:
\begin{enumerate}
\item[(a)] The degree of each $(r-1)$-set is $\gamma s^2 \pm C s$
\item[(b)] $\alpha_{(r-1)}=2\gamma$
\item[(c)] The weights satisfy $0\le \alpha_{\x}\le 4\gamma \ell^{-2}$ for all $\x\in \PP_{r-1}\setminus \{(r-1)\}$.
\end{enumerate}
\end{lemma}

\begin{proof} We shall prove this statement by induction.  For each $1\le j\le r-1$ consider the statement that there exists a weighting such that (a) is satisfied for sets whose type $\x$ has $|\x|\le j$, (b) and (c) are satisfied, and $\alpha_{\x}=0$ for all $|\x|>j$. 

The base case ($j=1$) of this statement is proved by taking $\alpha_{(r-1)}=2\gamma$.  It follows from Lemma~\ref{lem:12} that the degree of a set of type $(r-1)$ is
\[
 \frac{1}{2}\alpha_{(r-1)}s^2\, +\, O(s)\, =\, \gamma s^2\, +\, O(s)\, ,
 \]
 as required.
 
 We shall also treat $j=2$ as a base case.  We extend the choice above, and so $\alpha_{(r-1)}=2\gamma$.  Note that a set of type $\x$ with $|\x|=2$ will not be contained in a part $V_i$, and so will not be contained in any of the edges of $\HH^{\alpha_{(r-1)}}_{(r-1)}$.  We define $\alpha_{\x}=\gamma /\binom{\ell-2}{2}$.  And so, again by Lemma~\ref{lem:12}, we have that a set of type $\x$ has degree
 \[
 \binom{\ell-2}{2}\alpha_{\x}s^2\, +\, O(s)\, =\, \gamma s^2 \, +\, O(s)\, .
 \]
  
For the induction step, suppose the weighting has been defined for $\|x\|\le j-1$ we shall show how to extend the weighting for $\alpha_{\x}$ with $|\x|=j$.  Fix $\x$ with $|\x|=j$.  We first focus on the abstract degree $d^{\valpha}_{\x}$ of $\x$.  We use Lemma~\ref{lem:abstractdegs}.  Recall that in Lemma~\ref{lem:abstractdegs} $\alpha(a)$ is defined to be $\max\{\alpha_{\y}:2\le |\y|\le a-1\}$.  By the induction hypothesis we may assume that $\alpha(j)\le 4\gamma\ell^{-2}$.  It follows that the contribution to $d^{\valpha}_{\x}$ from the weights $\alpha_{\y}$ with $|\y|\le j-1$ is at most we have that the contribution of all edges of type $\y^+$ where $|\y|<|\x|$ is at most
\[
r^4\ell \alpha(j) s^2 \, \le\, 4\gamma \ell^{-1} r^4 s^2\, \le\, \gamma s^2\, .
\]
Let $0\le \gamma_{\x}\le \gamma$ be such that $\gamma_{\x}s^2$ is the remaining degree necessary, i.e., $\gamma s^2$ minus the contributions of earlier $\y$.  We may define 
\[
\alpha_{\x} \, =\, \frac{\gamma_{\x}}{\binom{\ell-|\x|}{2}}\, .
\]
We observe that $\alpha_{\x}$ are at most $4\gamma \ell^{-2}$, and that the abstract degree $d^{\valpha}_{\x}$ is $\gamma s^2 +O(s)$.

Once the weights $\alpha_{\x}$ have been chosen for all $|\x|=j$ we have by Lemma~\ref{lem:setdegs} (applied with $\alpha'\le 4\gamma \ell^{-2}$) that all $(r-1)$ tuples have degree
\[
\gamma s^2\, +\, O(s)\, ,
\]
as required.
\end{proof}

By taking $\HH=\HH^{\valpha}$, where the weights $\valpha$ are given by Lemma~\ref{lem:weight}, we have produced a hypergraph which satisfies properties (i), and (ii) from the definition of $(r,C/s,\gamma)$-nice.  In fact $\HH$ also satisfies condition (iii).  To see this we consider two cases.  For an $r$-set $A$ contained in a part $V_i$ the only edges containing $A$ are those of type $(r+1)$, it follows that the degree of $A$ is at most $\alpha_{(r-1)}s=2\gamma s\le \gamma N$.  For any other $r$-set $A$ it is contained in edges of at most $r$ of the hypergraphs $\HH_{\x^+}^{\alpha_{\x}}$ which make up $\HH$.  In each case to extend $A$ to an edge requires a choice of a part (at most $\ell$ possibilities) and a choice of a vertex to complete an $\alpha_{\x}$-good edge (at most $\alpha_{\x}s\le 4\gamma s/\ell^2$ possibilities).  It follows that the degree of $A$ is at most $4r\gamma s/\ell \le \gamma N$.  

All that remains is to prove that
\[
c^{\HH}_r\, \ge \, \gamma N^{r+1}/\ell^{r+1}\, .
\]
This is achieved by this final lemma.

\begin{lemma} Let $\HH=\HH^{\valpha}$ for some sequence of weights $\valpha=(\alpha_{\x}:\x\in \PP_{r-1})$ with $\alpha_{(r-1)}=2\gamma$ and
\[
0\le \alpha_{\x}\le 4\gamma \ell^{-2}
\]
for all $\x\in \PP_{r-1}\setminus \{(r-1)\}$.  Then
\[
c^{\HH}_r\, \ge \, \frac{\gamma s!}{r!}\, =\, \frac{\gamma N^{r+1}}{r!\ell^{r+1}}\, .
\]
\end{lemma}

\begin{proof} The hypergraph $\HH=\HH^{\valpha}$ is a union of hypergraphs 
\[
\bigcup_{\x\in \PP_{r-1}}\HH^{\alpha_x}_{\x^+}\, .
\]
Let us write $c^{\x}_r$ for the coefficient of $x^r$ in $Q^{\HH^{\alpha_x}_{\x^+}}(x)$.  By linearity we have that
\[
c^{\HH}_r\, =\, \sum_{\x \in \PP_{r-1}}c^{\x}_r\, .
\]
We shall prove that
\begin{enumerate}
\item[(a)] $\begin{displaystyle} c^{(r-1)}_r\, \ge\, \frac{r 2^r \gamma s^{r+1}}{(r+1)!}\phantom{\Bigg|}\end{displaystyle}$,
\item[(b)] $\begin{displaystyle} |c^{\x}_r|\, \le\, \frac{2^{r+2} \gamma s^{r+1}}{\ell}\phantom{\Bigg|}\end{displaystyle}$ if $|\x|=2$,
\item[(c)] $\begin{displaystyle} c^{\x}_r\, =\, 0\end{displaystyle}$ if $\|x\|\ge 3$.
\end{enumerate}
The required result will now follow from linearity.  We recall that $\ell= 4(r+1)!$ and that there are at most $r/2$ types $\x$ with $|x|=2$.   It follows that
 \begin{align*}
 c^{\HH}_r\, & \ge\, \frac{r 2^r \gamma s^{r+1}}{(r+1)!}\, -\, \frac{r2^{r+1} \gamma s^{r+1}}{\ell}\\
 & =\, \frac{r2^r\gamma s^{r+1} \, \left(\ell\, -\, 2(r+1)!\right)}{\ell(r+1)!}\\
 & \ge \, \frac{r2^{r-1}\gamma s^{r+1}}{(r+1)!}\\
 & \ge \, \frac{\gamma s^{r+1}}{r!}\, ,
 \end{align*}
as required.

We now prove (a), (b), and (c) stated above.

We begin with (a).  Since $\alpha_{(r-1)}=2\gamma$ the hypergraph $\HH^{2\gamma}_{(r+1)}$ in question contains $(2\gamma)$-good edges of type $(r+1)$, that is, $(2\gamma)$-good edges contained in one of the parts $V_i$.  The coefficient of $\eps^n$ in the corresponding polynomial
\[
Q^{\HH^{2\gamma}_{(r+1)}}(x) = \sum_{e \in E(\HH)} (1+2x)^{e_1} (1-x)^{e_2 + e_3}
\]
satisfies\footnote{We replace by $2$ by $7/4$ as we may lose something when we take the integer part.  The reader may easily verify the details.}
\begin{align*}
c^{(r-1)}_r\, \ge \, \frac{7}{4}\gamma \binom{s}{r+1}\left((r+1)2^r\, +\, 2(r+1)(-1)^r\right) \\
\ge\, \frac{r2^{r} \gamma s^{r+1}}{(r+1)!}\, ,
\end{align*}
as required.

We now prove (b).  A type $\x$ with $|\x|=2$ is of the form $\x=(j,r-j-1)$ for some $(r-1)/2\le j\le r-2$.  The corresponding $\x^+$ is $\x^+=(j,r-j-1,1,1)$.  And the corresponding hypergraph consists of $\alpha_{\x}$-good edges of type $\x^+$.  A contribution to the coefficient of $\eps^r$ is only made by edges which use all $3$ of the parts $V_1,V_2,V_3$ and one other part.  We obtain that
\begin{align*}
|c^{\x}_r|\, & \le\, 2\alpha_{\x}(\ell-3)\binom{s}{j}\binom{s}{r-j-1}s^2\left(2^j\, +\, 2^{r-1-j}\, +\, 2\right)\\
& \le \, \frac{8\gamma\ell s^{r+1} \, 2^{r-1}}{\ell^2 j! (r-1-j)!}\\
&\le \, \frac{2^{r+2} \gamma s^{r+1}}{\ell}\, ,
\end{align*}
as required.

Finally, we note that part (c) follows immediately from the fact that if $|\x|\ge 3$ then at most $r-1$ elements of any edge of the corresponding hypergraph lie in the union of $V_1\cup V_2\cup V_3$ and so the coefficient of $\eps^r$ is $0$.
\end{proof}

\end{document}